\documentclass[11pt,a4paper]{article}
\usepackage[T1]{fontenc}
\usepackage[utf8]{inputenc}
\usepackage{lmodern,amsmath,amssymb,amsthm,mathtools,microtype}
\usepackage{xcolor}
\usepackage[margin=27mm]{geometry}
\usepackage[hidelinks]{hyperref}
\usepackage{enumitem,booktabs,needspace}
\numberwithin{equation}{section}
\setlist{itemsep=3pt,topsep=5pt}
\newtheorem{theorem}{Theorem}[section]
\newtheorem{proposition}[theorem]{Proposition}
\newtheorem{lemma}[theorem]{Lemma}
\newtheorem{corollary}[theorem]{Corollary}
\theoremstyle{definition}
\newtheorem{definition}[theorem]{Definition}

\theoremstyle{remark}

\newcommand{\PP}{\mathbb P}
\newcommand{\EE}{\mathbb E}
\newcommand{\RR}{\mathbb R}
\newcommand{\ZZ}{\mathbb Z}
\newcommand{\Tr}{\operatorname{Tr}}
\newcommand{\LE}{\operatorname{LE}}
\newcommand{\ran}{\operatorname{Range}}
\newcommand{\inn}{\partial_{\mathrm{in}}}
\newcommand{\br}{\mathsf B}
\newcommand{\capn}{\operatorname{cap}}
\newcommand{\dist}{\operatorname{dist}}

\newcommand{\essinf}{\operatorname*{ess\,inf}}

\newcommand{\ind}{\mathbf 1}
\newcommand{\dd}{\,\mathrm d}
\renewcommand{\Im}{\operatorname{Im}}
\title{The boosted loop-erased random walk}
\author{Gefei Cai\thanks{Beijing International Center for Mathematical Research,
Peking University,
\href{mailto:caigefei1917@pku.edu.cn}{caigefei1917@pku.edu.cn}} \and Xinyi Li
\thanks{Beijing International Center for Mathematical Research,
Peking University,
\href{mailto:xinyili@bicmr.pku.edu.cn}{xinyili@bicmr.pku.edu.cn}} \and Daisuke Shiraishi\thanks{Graduate School of Informatics,
Kyoto University, \href{mailto:shiraishi@acs.i.kyoto-u.ac.jp}{shiraishi@acs.i.kyoto-u.ac.jp}} \and Aoteng Xia\thanks{School of Mathematics and Statistics, Central South University, \href{mailto:xiaaoteng@csu.edu.cn}{xiaaoteng@csu.edu.cn}}}
\date{\today
}
\begin{document}
\maketitle

\begin{abstract}

We consider a near-critical version of loop-erased random walk (LERW)
by boosting the weight of the underlying random walk according to its
length before loop-erasure. We construct the scaling limit in natural
parametrization in the two-dimensional unit disk and the three-dimensional
unit ball. The boosted scaling limit undergoes a phase transition in
mutual absolute continuity with the classical LERW scaling limit; in
the disk the transition occurs at boost rate $m^2=\pi^2/4$. As the boost
increases, the curve becomes macroscopically space-filling, in the sense
that unvisited cubes have side length $O(m^{-1})$, while its
Hausdorff dimension at each finite boost is unchanged. These results
are motivated by Makarov--Smirnov's questions on near-critical lattice
models. In the planar chordal setting, we also construct a continuum
loop tilt below the spectral threshold of the domain and calculate its
Loewner driving function.

\end{abstract}

\section{Introduction}\label{sec:intro}

Loop-erased random walk (LERW) is the random simple path obtained by
chronologically erasing the loops of a simple random walk. Introduced
by Lawler~\cite{LERW} as a model of self-avoiding walk, it is one of
the most tractable non-Gaussian critical models of statistical
physics. In dimensions four and higher, its
scaling limit is Brownian motion; see~\cite{LawlerGreen} and the
references therein. In dimension three, Kozma established the existence
of the scaling limit in~\cite{K} (see also \cite{Cai26Catcher,CLS26} for a new proof and characterization of the scaling limit). Many aspects of this limit have since
been understood, including its Hausdorff dimension, geometric properties
and convergence in natural parametrization; see
\cite{Shiraishi,ShiraishiDim,LSOnePoint,LS18,LSHolder,HLS}. In dimension
two, Lawler, Schramm and Werner~\cite{LSW} proved that LERW has a
conformally invariant scaling limit, described by Schramm--Loewner
evolution with parameter $2$, abbreviated as ${\rm SLE}_2$.
Subsequently, Yadin and Yehudayoff~\cite{YY} established universality,
and Lawler and Viklund~\cite{LV,LVrad} proved convergence in natural
parametrization in the chordal and radial cases.

As for other critical lattice models, it is natural to study
near-critical versions of LERW, obtained by perturbing it away from
criticality. As part of a program to study near-critical planar
lattice models, Makarov and Smirnov~\cite{MS} proposed massive LERW
(mLERW). Given a mesh size $\delta>0$ and mass parameter $m>0$, one
kills the underlying random walk with probability $m^2\delta^2$ at each
step, conditions it to reach its prescribed exit, and then erases its
loops. Its planar chordal scaling limit is massive ${\rm SLE}_2$,
whose Loewner driving function is explicit; see
\cite[Theorem~2.1]{MS} and Chelkak--Wan~\cite{CW} for a detailed proof.
Camia~\cite{Ca} introduced massive loop soups and explained their
connection with mLERW.

Near-critical behavior has also been studied for other planar lattice
models. Garban, Pete and Schramm~\cite{GPS} constructed near-critical
percolation from pivotal points. For Ising, Park~\cite{P} studied
thermally perturbed spin correlations, and Camia, Garban and
Newman~\cite{CGN} determined the magnetization exponent in a small
external field. More recently, Papon~\cite{PapIsing} proved convergence
of the magnetically perturbed interface to a law absolutely continuous
with respect to ${\rm SLE}_3$. For the near-critical random-field
Ising model, Huang, Papon and Xia~\cite{HPX26} proved quenched
convergence of the interface to an almost surely absolutely continuous
tilt of ${\rm SLE}_3$, while any subsequential outermost spin-loop
limit is almost surely singular with respect to ${\rm CLE}_3$.
For thermal perturbations,
Garban and Kupiainen~\cite{GK25} proved local singularity from the
critical limit in the FK quad-crossing topology. Mahfouf~\cite{Mah25}
studied random-bond perturbations through deformations of $s$-embeddings.
Shao~\cite{S} and Papon~\cite{Pa} studied the massive harmonic explorer
and its convergence to massive ${\rm SLE}_4$.
Berestycki and Haunschmid-Sibitz~\cite{B} related near-critical dimers
to massive ${\rm SLE}_2$, and Rey~\cite{Rey24} extended the convergence
of the associated massive LERW and tree models to isoradial graphs.
Berestycki, Mason and Rey~\cite{BMR26} identified the near-critical
dimer height limit with an electromagnetically tilted sine-Gordon
field. A related bosonization identity for near-critical Ising
correlations was proved by Park, Virtanen and Webb~\cite{PVW25}.

In this paper we introduce a different near-critical perturbation,
boosted LERW (bLERW): before the loops are erased, longer random-walk
paths are rewarded rather than penalized. Our main interest is the geometry of bLERW as the boost
increases. This is also motivated by a question of Makarov and Smirnov
in~\cite{MS}, which we briefly recall.

The planar $O(N)$ model is conjectured to have two critical regimes,
with interfaces described by ${\rm SLE}_\kappa$ and
${\rm SLE}_{\widetilde\kappa}$, where
\[
 \frac1\kappa+\frac1{\widetilde\kappa}=\frac12,
 \qquad \kappa\in[8/3,4],\quad\widetilde\kappa\in[4,8].
\]
For self-avoiding walk, the pair is $\kappa=8/3$ and
$\widetilde\kappa=8$; the space-filling property of supercritical
self-avoiding walk was proved in~\cite{SAW}. In
\cite[Question~4.12]{MS}, Makarov and Smirnov ask whether near-critical
models provide transition states between these regimes. Although
LERW, with $\kappa=2$, lies outside the range of that question, one may
ask the analogous question for bLERW. We prove that its scaling limit
exists in dimensions two and three and becomes macroscopically
space-filling as the boost tends to infinity. This gives a geometric
analogue of the proposed transition, while leaving open the limiting
law of the curves at large boost.

\subsection{The model}
\label{sec:model}

We first define the boosted walk in the range where it can be
normalized, and then introduce the random-walk loop measure used to
extend the model beyond this range.

Let $d\in\{2,3\}$, $D=B(0,1)\subset\RR^d$ and
$D_\delta=D\cap\delta\ZZ^d$. Let $S$ be simple random walk on
$\delta\ZZ^d$ started at $0$ and stopped at
$T_\delta=\inf\{j\geq0:S_j\notin D_\delta\}$.
Write $\gamma^\delta=\LE(S[0,T_\delta])$, including the exit edge,
and let $\eta^\delta$ be its linear interpolation with the natural
clock of Section~\ref{sec:natural-background}. We denote its trace
by $K^\delta$ and its law by $\PP_\delta$; the exit vertex is not
prescribed. The classical scaling limit $\eta$ of $\eta^\delta$ (see
Section~\ref{sec:natural-background}) has law $\PP$ and trace $K$.

For $m\geq0$, put $a_\delta=1+m^2\delta^2$ and assign a first-exit
walk $\omega$ the weight $(a_\delta/(2d))^{|\omega|}$. The partition
function is
\begin{equation}\label{eq:walk-partition}
 Z^{\rm RW}_{\delta,m}
 =\sum_{\omega:\,0\to\partial D_\delta}
       \left(\frac{a_\delta}{2d}\right)^{|\omega|}
 =\EE^0[a_\delta^{T_\delta}],
\end{equation}
where the sum is over walks stopped at their first exit. When this sum
is finite, bLERW is the chronological loop-erasure of the walk sampled
with these normalized weights. At $m=0$ we recover classical LERW.
We assign duration $\delta^2$ to each random-walk step.
Accordingly, throughout the paper, $d$-dimensional Brownian motion
has generator
\begin{equation}\label{eq:generator}
 L=\frac1{2d}\Delta,\qquad
 \Delta f=\sum_{j=1}^d\partial_j^2f.
\end{equation}
For an open set $U$, let $A_U=-\frac1{2d}\Delta_U$ be the
Dirichlet operator and write $\lambda_1(U)$ for its lowest eigenvalue.
The definition on nonsmooth slit domains is given in
Section~\ref{sec:loop-background}.


For a finite lattice domain $U$, let $p_{\delta,U}$ be the killed
transition matrix,
\begin{equation}\label{eq:killed-transition}
 p_{\delta,U}(x,y)=\frac1{2d}\ind\{|x-y|=\delta\},\qquad x,y\in U.
\end{equation}
Write $\rho(p_{\delta,U})$ for its spectral radius. The walk partition
function in~\eqref{eq:walk-partition} is finite precisely when
$a_\delta\rho(p_{\delta,D_\delta})<1$. In the continuum, the corresponding
exponential moment $\EE^0e^{m^2T_D}$ is finite precisely when
$m^2<\lambda_1(D)$, where $T_D$ is the Brownian exit time from $D$.
The question is whether the loop-erased model can be defined
beyond this range.

Random-walk loops enter because the loops erased from the walk
contribute to the weight of LERW. Their total
weight is expressed through the return traces
\begin{equation}\label{eq:trace-return}
 \Tr(p_{\delta,U}^{\,n})
 =\sum_{x\in U}\PP^x(S_n=x,\ n<T_U).
\end{equation}
The diagonal entry at $x$ is the total weight of length-$n$ walks
from $x$ that return to $x$ without leaving $U$. Random-walk loop
measure $\mu_{\delta,U}$ assigns a rooted loop of $n$ steps its walk
weight divided by $n$, and then forgets the time root. Thus the total
mass of length-$n$ loops is $\Tr(p_{\delta,U}^n)/n$.

A loop of $n$ steps has diffusive duration $t_\ell=n\delta^2$.
Boosting multiplies its weight by $a_\delta^n$, so the additional
loop mass is
\begin{equation}\label{eq:discrete-mass}
 \mathcal M_{\delta,m}(U)
 :=\int (a_\delta^{|\ell|}-1)\,\mu_{\delta,U}(\dd\ell)
 =\sum_{n\ge1}\frac{a_\delta^n-1}{n}
       \Tr(p_{\delta,U}^{\,n})\in[0,\infty].
\end{equation}
Here $|\ell|$ denotes the number of steps of the loop, and the
mass of the empty domain is zero. By the restriction property, the
loops in $U$ avoiding a set $K$ are exactly the loops in
$U\setminus K$; this is used in both changes of normalization below.

\subsection{Renormalization and the continuum weight}
\label{sec:intro-renormalization}

First suppose $a_\delta\rho(p_{\delta,D_\delta})<1$. For a fixed
simple first-exit path $\gamma$, let $K$ be its interior vertex trace,
that is, its vertices other than the exit vertex. Summing over the loops erased from the underlying walk gives
the unnormalized density relative to ordinary LERW:
\begin{equation}\label{eq:conditional-weight}
 w_{\delta,m}(\gamma) :=\EE^0[a_\delta^{T_\delta}\mid\gamma^\delta=\gamma]=a_\delta^{|\gamma|}
   \exp\!\left\{\mathcal M_{\delta,m}(D_\delta)
              -\mathcal M_{\delta,m}(D_\delta\setminus K)\right\}.
\end{equation}
The exponent is the excess mass of loops intersecting $K$, and
$|\gamma|$ counts the exit edge. This identity is proved in
Proposition~\ref{prop:agreement}.

The factor $\exp\{\mathcal M_{\delta,m}(D_\delta)\}$ does not depend
on $\gamma$ and cancels after normalization. Removing it leaves a
weight that involves only loops in the complement of the path and,
unlike \eqref{eq:conditional-weight}, makes sense even when the loop
mass in $D_\delta$ is infinite.

\begin{definition}[Extended discrete model]\label{def:extended}
For a simple first-exit path $\gamma$ with interior trace $K$, set
\[
 W_{\delta,m}(\gamma)=a_\delta^{|\gamma|}
   \exp\{-\mathcal M_{\delta,m}(D_\delta\setminus K)\},\qquad
 Z_{\delta,m}=\EE_{\PP_\delta}W_{\delta,m}(\gamma^\delta).
\]
Use $\exp(-\infty)=0$. If $Z_{\delta,m}>0$, the density
$W_{\delta,m}/Z_{\delta,m}$ defines $\PP_\delta^{(m)}$.
\end{definition}
Since there are finitely many simple paths, $Z_{\delta,m}<\infty$;
positivity of $Z_{\delta,m}$ for small $\delta$ is part of
Theorem~\ref{thm:all-boosts}.

This is the first change of normalization: it extends the discrete
loop-erased model beyond the range in which the walk itself can be
normalized. However, it does not yet produce an expression with a
nontrivial continuum limit.

Indeed, formally replacing random-walk loop measure in
$\mathcal M_{\delta,m}(D_\delta\setminus K)$ by Brownian loop
measure gives
\[
 \int_0^\infty(e^{m^2t}-1)H_{D\setminus K}(t)\frac{\dd t}{t},
 \qquad H_U(t):=\int_U p_U(t,x,x)\dd x,
\]
where $p_U$ is the killed Brownian transition density for the generator
in~\eqref{eq:generator}. Thus $H_U$ is the continuum counterpart of
the return trace in~\eqref{eq:trace-return}. For the classical
limiting trace, $D\setminus K$ has the full volume of $D$, so the
short-time contribution behaves like
\[
 (e^{m^2t}-1)H_{D\setminus K}(t)\frac{\dd t}{t}
 \asymp m^2t^{-d/2}\dd t,\qquad t\downarrow0,
 \qquad m>0.
\]
This is not integrable at zero for $d\in\{2,3\}$, so the direct
Brownian counterpart of $W_{\delta,m}$ vanishes even when the long-loop
contribution is finite. We therefore change the normalization a second
time before passing to the limit.

Fix $T>0$, and write $\mathcal M_{\delta,m}^{\le T}$ and
$\mathcal M_{\delta,m}^{>T}$ for loop mass restricted by the diffusive
duration. For each fixed mesh the short-loop mass of the whole
domain is finite, and multiplying $W_{\delta,m}$ by its exponential
gives
\begin{equation}\label{eq:intro-compensated-discrete}
 \begin{split}
 Y_{\delta,m,T}(\gamma)
 &\;:=e^{\mathcal M_{\delta,m}^{\le T}(D_\delta)}
       W_{\delta,m}(\gamma)\\
 &\;\;\,=a_\delta^{|\gamma|}\exp\!\left\{
      \mathcal M_{\delta,m}^{\le T}(D_\delta)
     -\mathcal M_{\delta,m}^{\le T}(D_\delta\setminus K)
     -\mathcal M_{\delta,m}^{>T}(D_\delta\setminus K)
   \right\}.
 \end{split}
\end{equation}
This multiplication leaves the normalized discrete law unchanged.
The first two terms in the exponent count short loops intersecting
$K$, and the last counts long loops avoiding $K$. We add back only the
short-loop mass of the whole domain, since its long-loop mass may be
infinite at the boosts under consideration. This cancels the
small-loop divergence without reintroducing the large-time
obstruction.

We now take the continuum counterpart of this rewritten expression.
Under diffusive scaling,
\[
 a_\delta^{|\ell|}
 =(1+m^2\delta^2)^{t_\ell/\delta^2}\longrightarrow e^{m^2t_\ell},
\]
and random-walk loop measure approaches Brownian loop measure.
Moreover, since the growth exponent of LERW is below two,
$a_\delta^{|\gamma^\delta|}\to1$ in probability.
These correspondences suggest the following candidate weight.
Let $\mu_D$ be Brownian loop measure for the generator
\eqref{eq:generator}, with the same diffusive duration $t_\ell$.
Define
\begin{align}
 I_{m,T}(K)
 &=\int_{\{t_\ell\le T,\ \ell\cap K\ne\varnothing\}}
       (e^{m^2t_\ell}-1)\,\mu_D(\dd\ell),\label{eq:I}\\
 J_{m,T}(K)
 &=\int_{\{t_\ell>T,\ \ell\cap K=\varnothing\}}
       (e^{m^2t_\ell}-1)\,\mu_D(\dd\ell),\label{eq:J}\\
 Y_{m,T}(K)&=\exp\{I_{m,T}(K)-J_{m,T}(K)\}.\label{eq:Y}
\end{align}
Here $K$ is the spatial trace and $T$ is a loop-duration cutoff.
We will prove that $I_{m,T}(K)$ is finite for the classical
limiting curve; if $J_{m,T}(K)=\infty$, we set $Y_{m,T}(K)=0$.
For $m=0$ the weight is identically one. Changing the cutoff from
$T$ to $T'>T$ gives
\begin{equation}\label{eq:cutoff-change}
 Y_{m,T'}(K)=Y_{m,T}(K)
 \exp\!\left\{\int_T^{T'}(e^{m^2t}-1)H_D(t)\frac{\dd t}{t}\right\}.
\end{equation}
The multiplier is finite and independent of the curve, so any
normalized law obtained from this weight is independent of $T$.

It remains to justify these formal correspondences: to control
arbitrarily short and long loops, to preserve their intersections with
the random curve, and to show that the limiting weight has a finite and
strictly positive mean. The main theorem establishes these facts and
identifies the candidate as the scaling limit of the extended discrete
law.

\subsection{Main results}

We now state our results for the radial models defined above, in
which the exit point is not prescribed. The first theorem shows that
the weight~\eqref{eq:Y} defines a probability law at every finite
boost, and that this law is the scaling limit of the extended discrete
model in natural parametrization.

\par\nopagebreak[4]
\begin{theorem}\label{thm:all-boosts}
In either dimension, for every fixed $m\geq0$, the extended model of
Definition~\ref{def:extended} is well defined for all sufficiently small
$\delta$ and converges in natural parametrization to
\[
 \PP^{(m)}(\dd\eta)=
 \frac{Y_{m,T}(K)}{\EE_{\PP}Y_{m,T}(K)}\,\PP(\dd\eta),
 \qquad 0<\EE_{\PP}Y_{m,T}<\infty.
\]
The law is independent of $T>0$, and
\[
 Y_{m,T}(K)>0\quad\Longleftrightarrow\quad
 \lambda_1(D\setminus K)>m^2
 \qquad\PP\text{-almost surely}.
\]
For $m^2<\lambda_1(D)$ this is the scaling limit of the loop erasure of
the normalized boosted walk in \eqref{eq:walk-partition}.
\end{theorem}

When $m^2<\lambda_1(D)$, the long-loop mass of the whole domain is
finite. Restoring this path-independent factor gives a simpler density,
expressed through the loops that hit the curve.

\par\nopagebreak[4]
\begin{theorem}\label{thm:subcritical}
For $0\leq m^2<\lambda_1(D)$, the limiting law satisfies
\begin{equation}\label{eq:subcritical-density}
 \frac{\dd\PP^{(m)}}{\dd\PP}(\eta)
 =\frac{\exp\{\Phi_m(K)\}}{\EE^0 e^{m^2T_D}},\qquad
 \Phi_m(K)=\int(e^{m^2t_\ell}-1)
 \ind\{\ell\cap K\ne\varnothing\}\,\mu_D(\dd\ell),
\end{equation}
where $T_D$ is the exit time of Brownian motion from $D$, started at $0$.
The integral is finite almost surely, the laws are mutually absolutely
continuous, and the density belongs to $L^p(\PP)$ whenever $p>1$ and
$pm^2<\lambda_1(D)$.
\end{theorem}

Our next result describes the phase transition in mutual absolute
continuity, in relation to \cite[Question~4.5]{MS} on absolute
continuity under near-critical perturbations. For massive LERW in the
plane, the perturbed and ordinary laws are equivalent for every
mass~\cite[Lemma~2.9]{CW}; boosting behaves differently.

\par\nopagebreak[4]
\begin{corollary}
\label{cor:phase-space}
For every $m\geq0$, $\PP^{(m)}\ll\PP$. In dimension three,
\[
 \PP\ll\PP^{(m)}\quad\Longleftrightarrow\quad
 m^2\leq\lambda_1(D).
\]

In dimension two,
\[
 \PP\ll\PP^{(m)}\quad\Longleftrightarrow\quad m^2\le\essinf_K \lambda_1(D\setminus K)=\pi^2/4.
\]

\end{corollary}

Thus, in dimension two and for $m^2\in[\lambda_1(D),\pi^2/4]$,
where $\lambda_1(D)\approx1.45$ and $\pi^2/4\approx2.47$, the Brownian
exponential moment $\EE^0e^{m^2T_D}$ is infinite, yet the extended law
is still mutually absolutely continuous with respect to the law of
${\rm SLE}_2$.

As $m$ increases, the curve becomes space-filling in the following
macroscopic sense.
\par\nopagebreak[4]
\begin{corollary}\label{cor:space-filling}
In either dimension, for every $m>0$, $\PP^{(m)}$-almost surely every
open cube $Q\subset D\setminus K$ satisfies
\begin{equation}\label{eq:empty-cube}
 \operatorname{side}(Q)<\frac{\pi}{\sqrt2\,m}.
\end{equation}
Consequently, for every $\ell>0$, if $m>M(\ell):=\pi/(\sqrt2\,\ell)$,
then $D\setminus K$ contains no cube of side length at least $\ell$,
$\PP^{(m)}$-almost surely.
\end{corollary}

Since $\PP^{(m)}\ll\PP$, the curve has the Hausdorff dimension of
the classical scaling limit at every finite boost: $5/4$ in dimension
two and $\beta\in[3/2,5/3]$ in dimension three\footnote{The classical lower bound $\beta>1$ was recently improved to
$\beta>1.0001$ by Liu~\cite[Theorem~1.1]{Liu26LERW}, and then to
$\beta\ge3/2$ by Cai~\cite[Theorem~1.3]{Cai26Ancestor}.}, where $\beta$ is the
three-dimensional LERW growth exponent~\cite{Shiraishi,ShiraishiDim}.
Thus macroscopic holes disappear as $m\to\infty$, while the dimension
does not change.

To compare the two perturbations further, in
Section~\ref{sec:chordal} we construct a continuum chordal bLERW law
for $m^2<\lambda_1(D)$ and compute its Loewner driving function
throughout this range. The construction uses the loop-decoration theorem of
Berestycki--Sauzedde~\cite{BS26}. For earlier work on reconstructing planar Brownian motion from
${\rm SLE}_2$ and Brownian loop-soups, see also
Ambrosio~\cite{Amb20}.
The driving-function computation follows the strategy of~\cite{MS,CW},
using a Green-kernel comparison that is uniform over the growing
slits throughout the subcritical range.

Our method also yields convergence of three-dimensional massive
LERW in the ball in natural parametrization. Here the per-step factor
is $1-m^2\delta^2$ and the walk is conditioned to exit before it is
killed; the limiting density is \eqref{eq:subcritical-density} with
$m^2$ replaced by $-m^2$. The finite-window and short-loop estimates
apply after the same sign change, and the long-loop tail is bounded by
the unweighted loop mass. Since the unnormalized weights are at most
one, dominated convergence gives the result for every fixed $m\ge0$.

\subsection{Proof sketch}
\label{sec:outline}

After the two changes of normalization, it suffices to prove
convergence and uniform integrability of the weights, whose logarithms
split as
\[
 \log Y_{\delta,m,T}
 =|\gamma^\delta|\log a_\delta+I_{\delta,m,T}-J_{\delta,m,T},
 \qquad |\gamma^\delta|\log a_\delta\longrightarrow0.
\]
We split the loop terms according to duration: short loops, loops
with duration in a fixed interval, and long loops. On a fixed
duration interval bounded away from zero, the coupling and hittability
results of~\cite{SS} preserve whether a loop intersects the path.
For short loops, the covering estimate controls $I_{\delta,m,T}$,
and its exponential-moment bound gives uniform integrability.
For long loops, the spectrum of the curve complement controls
$J_{\delta,m,T}$: a strict spectral gap gives a summable tail,
whereas divergent limiting mass gives zero weight. An independent
support construction ensures $\EE_{\PP}Y_{m,T}>0$. These inputs allow
us to pass to the normalized densities; their positivity criterion
then gives the absolute-continuity assertions and the bounds on
unvisited cubes.

A technical input of independent interest is the fixed-scale
box-count estimate of Theorem~\ref{thm:covering-input}, which is
uniform in both the mesh and the box size and has the sharp covering
powers $5/4$ in dimension two and $\beta$ in dimension three. For the
loop-weight argument any exponent below two suffices; the
corresponding covering content has exponential moments of every
order, uniformly in the mesh.

The proof follows the moment method of Barlow--Masson~\cite{BM} and
Shiraishi~\cite{Shiraishi}, with points replaced by boxes. The new
difficulty is that a visit of the walk to a box may later be erased.
We therefore mark, in each box, the last visit of the walk to the
first vertex of the box on the final LERW. The resulting decomposition
retains the condition that the incoming erased arm avoids the box, so
that the sharp hitting estimates of~\cite{LV} in dimension two and
of~\cite{LS18,HLS} in dimension three can be applied at each box.

\subsection{Further questions and organization}
\label{sec:further-questions}
We now briefly discuss two future directions and the organization of this work.

What can be said about boosted ${\rm SLE}_2$ as $m\to\infty$? The
disappearance of macroscopic holes
does not specify how the curve visits the domain, and it would be
interesting to identify a limiting law in a topology that retains
this information. Extending the driving-function computation to
larger boosts may help here.

In dimensions four and higher the question is different. There
ordinary LERW has a Brownian scaling limit, and Brownian traces have
zero Newtonian capacity~\cite[Chapters~8--9]{MP}. Above the spectral
threshold of the domain, however, the extended discrete weight forces
the path to raise the lowest eigenvalue of its complement. For $D=B(0,1)\subset\RR^d$ and the boost
$a_\delta=1+m^2\delta^2$ with fixed $m^2>\lambda_1(D)$, what are the
correct path-length scale and limiting object? In dimension $d\ge5$,
does the number of retained edges grow like
$\delta^{-(d-2)+o(1)}$, and what logarithmic correction arises in
dimension four? It is also natural to ask whether a rescaled
equilibrium measure, or the resulting killing potential, describes
this regime. The spectral constraint alone does not determine the
length or the limiting geometry.

Section~\ref{sec:background} fixes the standard path, operator and
loop-measure conventions. Section~\ref{sec:construction} proves the
loop-erasure and compensation identities used in the introduction
and states the inputs for convergence. Sections~\ref{sec:transfer}
and~\ref{sec:spectral-geometry} prove convergence and the geometric
consequences. Section~\ref{sec:chordal} treats the continuum chordal
law and its driving function. The box-count estimate is proved in Section~\ref{sec:sharp-box-counting}, and the
spectral support construction is given in Appendix~\ref{app:spectral}.

\medskip
\noindent\textbf{AI statement.} The main body of this work was developed in 2022. Commercial large language models were used to assist with language polishing and to simplify the proof of Theorem~\ref{thm:covering-input}.

\smallskip

\noindent\textbf{Acknowledgments.}
X.L. thanks Nathanael Berestycki, Dmitry Chelkak, Nina Holden, Xin Sun, Yilin Wang and Yizheng
Yuan for inspiring discussions. G.C. and X.L. are supported by Beijing Natural Science Foundation (JQ26001) and the
National Key R\&D Program of China (No.~2021YFA1002700). D.S. was supported in part by JSPS KAKENHI Grant Numbers JP26K00609 and JP23K22399.

\section{Notation and preliminaries}\label{sec:background}
This section collects standard background: path notation and loop
erasure, natural parametrization and the classical scaling limits,
Dirichlet operators, and the conventions for loop measures.

Write $B(x,r)=\{y:|y-x|<r\}$ and
$Q(x,r)=\{y:\|y-x\|_\infty<r\}$; on a specified lattice we mean their
intersection with that lattice. For $U\subset\delta\ZZ^d$,
\[
 \partial U=\{y\notin U:|x-y|=\delta\text{ for some }x\in U\},\quad
 \partial_{\rm in}U=\{x\in U:|x-y|=\delta\text{ for some }y\notin U\}.
\]
The symbols $\PP^x,\EE^x$ refer to the ordinary walk from $x$, or to
Brownian motion when specified. Constants $c,C$ may change from
line to line; subscripts indicate their dependence on parameters.

\subsection{Paths and stopping indices}\label{sec:path-conventions}
We fix the notation for path segments.

For $\omega=(\omega_0,\ldots,\omega_n)$, intervals include both
endpoints unless written $[i,j)$, and $\oplus$ concatenates paths with
the common endpoint counted once. Define
\begin{align}
 \tau_B&=\inf\{0\le j\le n:\omega_j\in B\},&
 \tau_B^+&=\inf\{1\le j\le n:\omega_j\in B\},
 \label{eq:path-hitting-times}\\
 T_U&=\inf\{0\le j\le n:\omega_j\notin U\},&
 L_B&=\sup\{0\le j\le n:\omega_j\in B\}.
 \label{eq:path-last-time}
\end{align}
Here $\inf\varnothing=\infty$, $\sup\varnothing=-\infty$; infinite
paths use all nonnegative indices and Brownian paths use continuous time.
We indicate the process in the notation when needed.

\subsection{Loop erasure}
We recall chronological loop erasure in terms of last visits and state the domain Markov property used below.

We use $S,Z$ for random walks, $\omega$ for a generic finite raw path, and $\eta,\gamma$ for simple or loop-erased paths. Subscripts and hats distinguish related paths. For $\omega=(\omega_0,\ldots,\omega_n)$, write
\[
 p(\omega)=\prod_{j=0}^{n-1}p(\omega_j,\omega_{j+1}),\qquad
 \operatorname{Range}(\omega)=\{\omega_0,\ldots,\omega_n\},
\]
where $p(u,v)$ is the one-step transition probability and an empty product is one.

Set $\gamma_0=\omega_0$ and recursively
\begin{equation}\label{eq:chronological-indices}
 \begin{split}
 s_0&=\max\{0\le j\le n:\omega_j=\gamma_0\},\\
 \gamma_{k+1}&=\omega_{s_k+1},\qquad
 s_{k+1}=\max\{s_k+1\le j\le n:\omega_j=\gamma_{k+1}\}
 \quad(s_k<n).
 \end{split}
\end{equation}
Then $\LE(\omega)=(\gamma_0,\ldots,\gamma_\ell)$, where
$\ell=\inf\{k:s_k=n\}$. For a first-exit path, $\gamma_\ell$ is the
exit vertex and $\ell$ counts edges including the exit edge.
For a transient infinite walk the maxima are taken over all future
indices; this defines infinite LERW in dimension three.

The first entrance of the \emph{final} loop erasure into $B$ is
\begin{equation}\label{eq:surviving-first-entrance}
 j_B=\inf\{j:\gamma_j\in B\},\qquad y_B=\gamma_{j_B},\qquad
 t_B=s_{j_B}\quad\text{on }\{j_B<\infty\}.
\end{equation}
Thus $t_B$ is the last visit of the walk to the surviving vertex
$y_B$, which in general differs from both $\tau_B$ and $L_B$. Since  \(t_B\) is generally not a stopping time, the last-visit decompositions
below are identities between sums over finite paths. Similarly, loop
erasure does not commute with reversal pathwise; we only use the
reversibility of the law of LERW \cite[Section~11.2]{LL}.

We recall the domain Markov property of $\gamma=\LE(S[0,T_U])$. Given
$\gamma[0,k]=\lambda=(v_0,\ldots,v_k)$ with $v_k\in U$, start
$\widetilde S$ at $v_k$, and let
$\widetilde T_U=\inf\{j:\widetilde S_j\notin U\}$ and
$\sigma_\lambda=\inf\{j\ge1:\widetilde S_j\in\lambda[0,k)\}$.
Then~\cite[Section~11.2]{LL}
\begin{equation}\label{eq:ordinary-dmp}
 \mathcal L(\gamma[k,|\gamma|]\mid\gamma[0,k]=\lambda)
 =\mathcal L\bigl(\LE(\widetilde S[0,\widetilde T_U])
                      \mid\widetilde T_U<\sigma_\lambda\bigr).
\end{equation}
In the conditioning the walk must avoid the retained path before
its tip, but may return to the tip. Equivalently, one may forbid
positive-time returns to the tip: the loops at the tip contribute the
common factor $G_{\delta,U\setminus\lambda[0,k)}(v_k,v_k)$, which
cancels after normalization. The ordinary Green function is
\begin{equation}\label{eq:ordinary-green}
 G_{\delta,U}(x,y)=\EE^x\sum_{j=0}^{T_U-1}\ind\{S_j=y\}
     =\sum_{j\ge0}\PP^x(S_j=y,\ j<T_U).
\end{equation}
It includes the time-zero visit; on the unit lattice we abbreviate it
as $G_U$. The Green bridges of Section~\ref{sec:sharp-box-counting}
are path measures normalized by this sum.

\subsection{Natural parametrization and the classical limit}
\label{sec:natural-background}

We now specify the natural time assigned to each LERW edge and recall the convergence results for our two radial settings.

The raw walk runs on the diffusive clock, with time $\delta^2$ per
step, whereas LERW is parametrized by its natural length:
\[
 c_\delta=\begin{cases}c_*^{-1}\delta^{5/4},&d=2,\\
                         \delta^\beta,&d=3,\end{cases}
 \qquad \eta^\delta(jc_\delta)=\gamma_j^\delta,
 \qquad t_{\eta^\delta}=c_\delta|\gamma^\delta|,
\]
with linear interpolation. Here $c_*$ is the planar normalization
in~\cite{LVrad}, and  $\beta\in[3/2,5/3]$ is the three-dimensional LERW growth exponent introduced above. The limiting curves are parametrized by the
corresponding constant multiples of Minkowski content. For
finite-duration continuous curves we use
\[
 \rho(\eta,\widetilde\eta)=|t_\eta-t_{\widetilde\eta}|+
 \sup_{0\le s\le1}|\eta(st_\eta)-\widetilde\eta(st_{\widetilde\eta})|.
\]
This metric has the same convergent sequences as the
Lawler--Viklund metric, which allows time changes uniformly close to
the identity.

In dimension two we need the following version of the radial
convergence of~\cite{LVrad}, in which the exit point is not prescribed.

\par\nopagebreak[4]
\begin{lemma}
\label{lem:planar-free-natural}
In the unit disk, $\eta^\delta$ converges to the reversal of radial
${\rm SLE}_2$ from a uniformly chosen boundary point to $0$, in natural
parametrization by $5/4$-dimensional Minkowski content.
\end{lemma}

\begin{proof}
We first apply the known natural-parametrization convergence theorem to the standard inner approximation of the disk, and then show that replacing this approximation by \(D_\delta\) changes the LERW only by terminal pieces of vanishing natural duration.
For the inner union-of-squares approximation, condition on the exit
edge and reverse the LERW. Apply~\cite[Theorem~1.1]{LVrad}; convergence
is uniform over the starting edges by compactness of the circle and
rotation-continuity of radial ${\rm SLE}_2$. Averaging against the exit
distribution, which tends to uniform arclength, gives the claim for
that approximation. To pass to $D_\delta$, couple the two walks until
exit. The boundaries of the two approximating domains differ by $O(\delta)$, so the discrete Beurling
estimate~\cite{Kesten1987,LLBeurling}
confines the added walk segment to an $\varepsilon$-neighborhood of
the exit point with probability tending to one. It can erase only a
terminal part of the old LERW; natural convergence and the fact that
the limiting arc touches the boundary only at its endpoint imply that the natural duration of the erased terminal segment tends to zero as $\varepsilon\downarrow0$. On the confinement
event the added suffix lies in the boundary strip; the estimate
\cite[Theorem~5.6(1) and Lemma~3.13]{BM} gives expected natural LERW
time at most
$C_\theta\varepsilon^{5/4-\theta}$ for $0<\theta<5/4$.
For $N=\delta^{-1}$, the bound $N^2\operatorname{Es}(N)=O(N^{5/4})$
follows from natural-time
tightness for the inner approximation and the lower-tail bound
\cite[Theorem~1.2 and Proposition~6.2(2)]{BM}.
The common prefixes and the vanishing terminal pieces give convergence
in the natural curve metric, also for arbitrary meshes.
\end{proof}

In dimension three we use the following result. Convergence in natural parametrization along dyadic meshes was
proved by Li--Shiraishi~\cite{LS18}; for arbitrary meshes, in the
present setting (unit ball, start at $0$, free exit), see
\cite[Theorem~1.7]{HLS}.

\par\nopagebreak[4]
\begin{proposition}\label{prop:3d-natural}
In the unit ball, $\eta^\delta$ converges in the natural curve topology
for arbitrary meshes $\delta\downarrow0$.
\end{proposition}

\subsection{Dirichlet operators and heat traces}
\label{sec:loop-background}\label{sec:operators}

We define the discrete operators, the Dirichlet operators on curve
complements, and the heat traces used for loop masses. Recall that the generator
is normalized as in~\eqref{eq:generator}.

Extending $f:U\to\RR$ by zero as $\widetilde f$, our discrete
operators are
\begin{equation}\label{eq:discrete-generator}
 L_{\delta,U}f(x)=\frac1{2d\delta^2}
   \sum_{|y-x|=\delta}[\widetilde f(y)-f(x)],\qquad
 A_{\delta,U}=-L_{\delta,U}=\delta^{-2}(I-p_{\delta,U}).
\end{equation}
For any open $U\subset D$, the positive continuum operator
$A_U=-\Delta_U/(2d)$ is associated with the closed Dirichlet form
\begin{equation}\label{eq:dirichlet-form}
 \mathcal E_U(f,g)=\frac1{2d}\int_U\nabla f\cdot\nabla g\dd x,
 \qquad f,g\in H^1_0(U).
\end{equation}
Here $H^1_0(U)$ is the $H^1$ closure of $C_c^\infty(U)$, with zero
extension to $D$. This defines Dirichlet boundary conditions on
arbitrary slits; see~\cite[Section~2]{BBcapacity} for the capacitary
formulation.
Compact embedding gives eigenvalues
$0<\lambda_1(U)\le\lambda_2(U)\le\cdots$, counted with multiplicity, and
\begin{equation}\label{eq:rayleigh-convention}
 \lambda_1(U)=\inf_{\substack{f\in H^1_0(U)\\\|f\|_2=1}}
                  \frac1{2d}\int_U|\nabla f|^2\dd x.
\end{equation}
The killed Brownian density $p_U(t,x,y)$ has semigroup $e^{-tA_U}$
and heat trace
\begin{equation}\label{eq:heat-trace-definition}
 H_U(t)=\Tr(e^{-tA_U})=\int_U p_U(t,x,x)\dd x
                    =\sum_{j\ge1}e^{-t\lambda_j(U)}.
\end{equation}
We put $H_\varnothing=0$, $\lambda_1(\varnothing)=\infty$.

\subsection{Loop measure conventions}
\label{sec:loop-measures}

We fix the conventions for rooted and unrooted loops, for both the
random-walk and the Brownian loop measure, with the diffusive clock
used throughout.

A rooted loop
$\omega=(x_0,\ldots,x_n=x_0)$ has mass
$n^{-1}\prod_{j=0}^{n-1}p_{\delta,U}(x_j,x_{j+1})$; forgetting the
root defines unrooted loop measure. The total length-$n$ mass is
$\Tr(p_{\delta,U}^n)/n$. A $k$-fold repetition has $n/k$ distinct
roots and therefore mass equal to its walk weight divided by $k$
\cite[Sections~9.2--9.5]{LL}.
Its diffusive duration is $t=n\delta^2$, so
$a_\delta^n=e^{\nu_\delta t}$ with
$\nu_\delta=\delta^{-2}\log a_\delta\to m^2$.
Vertex traces and their interpolations differ by $O(\delta)$;
matrix restrictions always remove vertices.

Brownian loop measure, with the same clock, is defined by
\begin{equation}\label{eq:bridge-loop-definition}
 \int F(\ell)\mu_U(\dd\ell)=\int_0^\infty\frac{\dd t}{t}
       \int_U p_U(t,x,x)\EE_{x,x}^{U,t}[F(\ell)]\dd x,
\end{equation}
for nonnegative root-independent functionals $F$, where
$\EE_{x,x}^{U,t}$ is the probability bridge of the killed process; see~\cite{LW04} and~\cite[Section~2]{SS}.
In particular,
\begin{equation}\label{eq:loop-trace}
 \mu_D\{t_\ell\in\dd t\}=H_D(t)\frac{\dd t}{t}.
\end{equation}
The restriction property identifies loops avoiding $K$ with loops in
$D\setminus K$, and therefore
\begin{align}
 \mu_D\{t_\ell\in\dd t,\ \ell\cap K=\varnothing\}
   &=H_{D\setminus K}(t)\frac{\dd t}{t},\nonumber\\
 \mu_D\{t_\ell\in\dd t,\ \ell\cap K\ne\varnothing\}
   &=[H_D(t)-H_{D\setminus K}(t)]\frac{\dd t}{t}.
 \label{eq:loop-hit-avoid-traces}
\end{align}
These identities express the integrals
\eqref{eq:I}--\eqref{eq:J} through heat traces; see
Section~\ref{sec:compensated-background}.

\section{Loop identities and inputs for convergence}
\label{sec:construction}

In this section, we prove the loop identities behind the construction in
Section~\ref{sec:intro-renormalization} and then state the covering and
support estimates used to pass to the limit. These estimates are used
in Section~\ref{sec:transfer} and proved in
Section~\ref{sec:sharp-box-counting} and Appendix~\ref{app:spectral}.

\subsection{The loop-erasure identity and the discrete spectral cutoff}
\label{sec:finite}

We first prove the loop-erasure identity
\eqref{eq:conditional-weight} and show that
Definition~\ref{def:extended} agrees with the boosted walk whenever the
latter can be normalized.

\par\nopagebreak[4]
\begin{proposition}\label{prop:agreement}
If $a_\delta\rho(p_{\delta,D_\delta})<1$, the extended law is the loop
erasure of the normalized walk in \eqref{eq:walk-partition}.
For a simple first-exit path $\gamma$ with interior trace $K$, its unnormalized
density relative to ordinary LERW is the conditional expectation
$w_{\delta,m}(\gamma)$ in \eqref{eq:conditional-weight}.

\end{proposition}
\begin{proof}

For $\gamma=(v_0,\ldots,v_l)$, put
$K=\{v_0,\ldots,v_{l-1}\}$ and
$U_j=D_\delta\setminus\{v_0,\ldots,v_{j-1}\}$.
For $c\in\{1,a_\delta\}$, let
$q(x,y)=c(2d)^{-1}\ind_{\{|x-y|=\delta\}}$, including on the exit edge,
and let $Q_U$ be its restriction to $U$. The standard loop-erasure identity
\cite[Proposition~9.5.1]{LL}, with the exit vertex absorbing, gives
\begin{equation}\label{eq:fp-finite-path-sum}
 \sum_{\LE(\omega)=\gamma}q(\omega)
 =q(\gamma)\prod_{j<l}[(I-Q_{U_j})^{-1}](v_j,v_j)
 =q(\gamma)\frac{\det(I-Q_{D_\delta\setminus K})}
                  {\det(I-Q_{D_\delta})}.
\end{equation}

The sum is over first-exit walks and converges by the spectral
assumption. Dividing the boosted and ordinary identities and expanding
the logarithms gives \eqref{eq:conditional-weight}. The factor
depending only on loops in $D_\delta$ cancels on normalization,
leaving $W_{\delta,m}$.
\end{proof}

Beyond this range the extended weight is still defined. The next
proposition identifies the paths with positive weight and expresses
the weight through determinants.

\par\nopagebreak[4]
\begin{proposition}\label{prop:finite-det}
Write $\rho(Q)$ for the spectral radius of a matrix $Q$.
For $m>0$, $W_{\delta,m}(\gamma)>0$ precisely when
$a_\delta\rho(p_{\delta,D_\delta\setminus K})<1$.
For every finite $U$ satisfying $a_\delta\rho(p_{\delta,U})<1$,
\begin{equation}\label{eq:finite-det}
 e^{-\mathcal M_{\delta,m}(U)}=
 \frac{\det(I-a_\delta p_{\delta,U})}{\det(I-p_{\delta,U})}.
\end{equation}
\end{proposition}
\begin{proof}
Use $\log\det(I-Q)=-\sum_{n\ge1}\Tr(Q^n)/n$ for $\rho(Q)<1$.
Bipartiteness gives zero odd traces and
$\Tr(p_{\delta,U}^{2k})\ge\rho(p_{\delta,U})^{2k}$, so at or beyond
the threshold the even terms in \eqref{eq:discrete-mass} diverge.
If $\rho(p_{\delta,U})=0$, the loop mass is zero.
\end{proof}

Combining \eqref{eq:finite-det} with the ordinary LERW path weight gives
\[
 \PP_\delta^{(m)}(\gamma)\ \propto\
 \left(\frac{a_\delta}{2d}\right)^{|\gamma|}
 \det(I-a_\delta p_{\delta,D_\delta\setminus K})
 \ind_{\{a_\delta\rho(p_{\delta,D_\delta\setminus K})<1\}}.
\]

\subsection{Exact compensation identities}
\label{sec:compensated-background}

We now write the short- and long-loop terms of the second
normalization~\eqref{eq:intro-compensated-discrete} explicitly and
verify it at each fixed mesh, including for paths of zero weight.

\par\nopagebreak[4]
\begin{lemma}\label{lem:discrete-comp}
For a simple exit path $\gamma$ with interior trace $K$, set
\begin{align*}
 I_{\delta,m,T}(K)&=\sum_{n\delta^2\le T}\frac{a_\delta^n-1}{n}
 [\Tr(p_{\delta,D_\delta}^n)-\Tr(p_{\delta,D_\delta\setminus K}^n)],\\
 J_{\delta,m,T}(K)&=\sum_{n\delta^2>T}\frac{a_\delta^n-1}{n}
          \Tr(p_{\delta,D_\delta\setminus K}^n).
\end{align*}
Then $Y_{\delta,m,T}(\gamma)
=a_\delta^{|\gamma|}\exp\{I_{\delta,m,T}(K)-J_{\delta,m,T}(K)\}$
is $W_{\delta,m}(\gamma)$ times the finite deterministic factor
$e^{C_{\delta,m,T}}$, where
\[
 C_{\delta,m,T}=\sum_{n\delta^2\le T}
       \frac{a_\delta^n-1}{n}\Tr(p_{\delta,D_\delta}^n).
\]
Hence the two weights define the same normalized law whenever the normalizing constant is positive. The identity also holds when the long-loop mass is infinite, in which case both weights vanish.
\end{lemma}
\begin{proof}
Split \eqref{eq:discrete-mass} at $n\delta^2=T$:
$\mathcal M_{\delta,m}(D_\delta\setminus K)
=C_{\delta,m,T}-I_{\delta,m,T}(K)+J_{\delta,m,T}(K)$.
Only the last term can be infinite. Exponentiation gives the claim.
\end{proof}

For the continuum weight  defined in
\eqref{eq:I}--\eqref{eq:Y}, the restriction identities
\eqref{eq:loop-hit-avoid-traces} give
\begin{align*}
 I_{m,T}(K)&=\int_0^T(e^{m^2t}-1)
       [H_D(t)-H_{D\setminus K}(t)]\frac{\dd t}{t},\\
 J_{m,T}(K)&=\int_T^\infty(e^{m^2t}-1)
       H_{D\setminus K}(t)\frac{\dd t}{t}.
\end{align*}
If $m^2<\lambda_1(D)$, adding the finite long-loop mass in $D$
recovers the uncompensated hitting expression:

\begin{equation}\label{eq:compensation-motivation}
 \Phi_m(K)=I_{m,T}(K)-J_{m,T}(K)
       +\int_T^\infty(e^{m^2t}-1)H_D(t)\frac{\dd t}{t}.
\end{equation}

For the relation to regularized determinants on smooth surfaces see
\cite{BL}; our definition uses heat traces directly on slit domains.

\subsection{Inputs for the convergence proof}
\label{sec:inputs}\label{sec:main-inputs}
We now state the covering estimate needed for integrability and the support estimate needed for a positive normalizing constant.

Fix a grid
$\mathcal Q_r=\{r(z+[0,1]^d):z\in\ZZ^d\}$ and write
\[
 N_r(K)=\#\{Q\in\mathcal Q_r:Q\cap K\ne\varnothing\}.
\]
Thus $N_r(K)$ measures the trace at resolution $r$.
Write
\begin{equation}\label{eq:b23}
 b_2=\frac54,\qquad b_3=\beta\in[3/2,5/3],
\end{equation}
where $\beta$ is the three-dimensional LERW growth exponent;
the sharp growth estimates and its identification with the scaling
limit's Hausdorff dimension are given in
\cite[Corollary~1.3]{LSOnePoint} and~\cite{ShiraishiDim}.
For $0<\alpha<2$, define
\begin{equation}\label{eq:box-content-definition}
 M_{\alpha,\delta}(K)=
 \sup_{\substack{r\in\{1,2^{-1},2^{-2},\ldots\}\\r\geq\delta}}
       r^\alpha N_r(K),\qquad
 M_\alpha(K)=\sup_{r\in\{1,2^{-1},2^{-2},\ldots\}}r^\alpha N_r(K).
\end{equation}
The restriction $r\ge\delta$ appears only in the discrete content.

The geometric input is the following box-counting estimate. The proof is given in
Section~\ref{sec:sharp-box-counting}.

\begin{theorem}
\label{thm:covering-input}
In each of the radial settings fixed in Section~\ref{sec:model},
there are constants $c,C>0$ and $\delta_0>0$ such that
\begin{equation}\label{eq:sharp-fixed-scale-input}
 \PP_\delta\{N_r(K^\delta)>u r^{-b_d}\}
 \le C e^{-cu},
 \qquad \delta\le r\le1,\quad u\ge1,
 \quad0<\delta<\delta_0.
\end{equation}
For every $b_d<\alpha<2$ and every $\lambda>0$,
\begin{equation}\label{eq:EB}
 \sup_{0<\delta<\delta_0}\EE_{\PP_\delta}
       \exp\{\lambda M_{\alpha,\delta}(K^\delta)\}<\infty.
\end{equation}
Here $K^\delta$ is the linearly interpolated trace of ordinary LERW.
The constants in the fixed-scale estimate are also uniform over
translations of the box grid.
\end{theorem}

The second input ensures that the limiting normalizer is positive. As it follows from classical estimates, we will postpone its proof to  Appendix~\ref{app:spectral}.

\begin{proposition}\label{prop:all-support}
For standard radial LERW in either dimension,
\begin{equation}\label{eq:all-support}
 \PP\{\lambda_1(D\setminus K)>M\}>0
 \qquad\text{for every finite }M.
\end{equation}
\end{proposition}

The following theorem, together with Theorem~\ref{thm:covering-input}
and Proposition~\ref{prop:all-support}, implies
Theorem~\ref{thm:all-boosts}. We will give its proof  (as well as that of Theorem~\ref{thm:all-boosts}) in
Section~\ref{sec:transfer}.
\begin{theorem}\label{thm:extension-intro}
In either of the two radial settings, suppose that \eqref{eq:EB} holds
for some $\alpha<2$. For every fixed $m\geq0$ such that
\begin{equation}\label{eq:positive-support}
 \PP\{\lambda_1(D\setminus K)>m^2\}>0,
\end{equation}
the extended discrete laws are well defined for all sufficiently small
$\delta$ and converge in natural parametrization to
$Y_{m,T}\PP/\EE Y_{m,T}$, with $0<\EE Y_{m,T}<\infty$.
The limiting law does not depend on $T$.
\end{theorem}

\section{Convergence of the boosted laws}\label{sec:transfer}
\label{sec:subcritical-proof}
 In this section we prove Theorem~\ref{thm:extension-intro} and Theorem~\ref{thm:all-boosts} by
showing that the compensated weights of
Lemma~\ref{lem:discrete-comp} converge to~\eqref{eq:Y}, using the inputs
of Section~\ref{sec:inputs}. In particular, loops with duration in a fixed window
are handled by a loop-soup coupling, short loops by the covering
estimate, and long loops by spectral decay.

\subsection{Loops in a fixed duration interval}\label{sec:fixed-duration}
On a fixed duration interval, we prove convergence of the loop masses hitting and avoiding the curve by combining loop-soup coupling with intersection stability.

\begin{lemma}\label{lem:finite-duration-loops}
In either radial setting, couple the ordinary naturally parametrized
curves so that $K^\delta\to K$ almost surely. For $0<u<v<\infty$ and bounded continuous
$f:[u,v]\to\RR$,
\begin{equation}\label{eq:window}
 \sum_{u\leq n\delta^2\leq v}\frac{f(n\delta^2)}n
 \Tr(p_{\delta,D_\delta\setminus K^\delta}^{n})
 \longrightarrow
 \int_u^v f(t)H_{D\setminus K}(t)\frac{\dd t}{t}
 \quad\text{in }L^1.
\end{equation}
The same holds for the loops hitting the curve, by subtraction from
the corresponding sums for $D_\delta$ and $D$.
\end{lemma}
\begin{proof}
Couple intensity-one loop soups $\mathcal L_\delta,\mathcal L$
independently of the curves using \cite[Theorem~2.2]{SS}, with duration
normalized by \eqref{eq:generator}; the planar approximation originates
in~\cite{RWLS}. On the fixed window, paired loops converge in duration
and parametrized trace, with probability tending to one. Their total
intensities are bounded by $C_D\int_{u/2}^{2v}t^{-d/2-1}\dd t$.
We first keep only loops of diameter greater than $r>0$.

Pairing alone does not preserve intersection with the curve. Let
$\mathcal N_\delta(h,r)$ be the event that some loop in
$\mathcal L_\delta$, contained in $D_\delta$ and of diameter greater
than $r$, avoids $K^\delta$ but comes within distance $h$ of it.
In the three-dimensional radial unit ball,
\cite[Proposition~5.2]{SS} gives, for some $\zeta>0$,
\begin{equation}\label{eq:lt-near-miss-bound}
 \limsup_{\delta\downarrow0}\PP(\mathcal N_\delta(h,r))
 \leq C(h^{1/2}+h^\zeta r^{-5}),\qquad 0<h<r.
\end{equation}

Its hypotheses are radial hittability of the curve and
independence of the loop soup.
In dimension two, the proof of \cite[Proposition~5.2]{SS} applies with
the planar Beurling estimate as its hittability input: the interpolated
set $K^\delta\cup D^c$ is connected and unbounded, so that input is
deterministic and uniform, including at the two endpoints. It gives
$\lim_{h\downarrow0}\limsup_{\delta\downarrow0}
\PP(\mathcal N_\delta(h,r))=0$ for each fixed $r>0$.

A limiting loop avoiding $K$ has positive distance from it. A loop
hitting $K$ whose partner avoids $K^\delta$ produces a near miss,
excluded by \eqref{eq:lt-near-miss-bound}, first as $\delta\downarrow0$
then $h\downarrow0$. Containment in $D$ is stable by sphere regularity
and the boundary-layer estimate \cite[Proposition~5.1]{SS}.
Let $X_{\delta,r}$ be the sum of $f(t_\delta(\ell))$ over loops in
$\mathcal L_\delta$ with duration in $[u,v]$, diameter greater than $r$,
contained in $D_\delta$ and avoiding $K^\delta$; define $X_r$ analogously.
The pairing
then gives $X_{\delta,r}\to X_r$ in probability, and the Poisson
second-moment formula gives
$\sup_\delta\EE|X_{\delta,r}|^2+\EE|X_r|^2
 \leq C_{D,u,v}\|f\|_\infty^2$.
Thus convergence is in $L^1$. Conditional expectation given
$\mathcal G=\sigma(K^\delta,K)$ identifies the desired intensities and
\begin{equation}\label{eq:lt-poisson-contraction}
 \EE\left|\EE[X_{\delta,r}-X_r\mid\mathcal G]\right|
 \leq\EE|X_{\delta,r}-X_r|\longrightarrow0.
\end{equation}

Loops of duration at least $u$ and diameter at most $r$ stay in a box
of side $Cr$ around their root. Splitting its killed heat kernel at
half-time bounds their intensity by
$C_D\int_u^v t^{-d/2-1}e^{-ct/r^2}\dd t\to0$; the analogous lattice
sum is uniform as $\delta\downarrow0$. This removes the diameter cutoff.
For boosted multipliers use the uniform convergence on $[u,v]$ of
\begin{equation}\label{eq:lt-window-multiplier}
 a_\delta^n-1=e^{\nu_\delta n\delta^2}-1,
 \qquad \nu_\delta=\delta^{-2}\log(1+m^2\delta^2)\longrightarrow m^2.
\end{equation}
The factor two in the even-step local central limit theorem cancels
the duration spacing $2\delta^2$, so \eqref{eq:window} also incorporates
lattice periodicity.
\end{proof}

\subsection{Short loops}\label{sec:short-loops}

For loops of duration $t$, we cover the curve by boxes of side
comparable to $\sqrt t$. Each box carries bounded return mass at that
scale, while the boost contributes a factor of order $t$; summability
therefore requires a covering exponent below two. Short loops are
treated similarly in~\cite[Proposition~6.13]{BS26}.

\par\nopagebreak[4]
\begin{lemma}
\label{lem:subcritical-short-loops}
For $0<\alpha<2$, the covering contents of Section~\ref{sec:inputs}
satisfy
\begin{align}
 H_D(t)-H_{D\setminus K}(t)
 &\le C M_\alpha(K)t^{-\alpha/2},&&0<t\le1,
 \label{eq:det-short}\\
 I_{m,T}(K)&\le C_{d,\alpha}m^2e^{m^2T}
 M_\alpha(K)T^{1-\alpha/2},&&0<T\le1,
 \label{eq:small-coefficient}\\
 I_{\delta,m,T}(K^\delta)&\le C_{d,\alpha}m^2e^{m^2T}
 M_{\alpha,\delta}(K^\delta)T^{1-\alpha/2},&&0<T\le1.
 \label{eq:discrete-small-coefficient}
\end{align}
\end{lemma}
\begin{proof}
For an $r$-box $Q$ and $r^2\le n\delta^2<4r^2$, reversal, the
Markov property at $\tau_Q$, and the Gaussian heat-kernel bound give
\[
 \sum_x\PP^x(S_n=x,\ S[0,n]\cap Q\ne\varnothing)
 \le Cn^{-d/2}\sum_x\PP^x(\tau_Q\le n/2)\le C.
\]
The last sum is $O(n^{d/2})$ by the maximal-displacement bound.
Discarding killing and summing over a cover yields
\begin{equation}\label{eq:lt-hit-trace-cover}
 0\le\Tr(p_{\delta,D_\delta}^n)
       -\Tr(p_{\delta,D_\delta\setminus K^\delta}^n)
 \le C N_r(K^\delta),\qquad r^2\le n\delta^2<4r^2.
\end{equation}
The identical bridge argument in physical time gives
$H_D(t)-H_{D\setminus K}(t)\le C N_r(K)$ for $r\asymp\sqrt t$,
hence \eqref{eq:det-short}. Integrate against
$(e^{m^2t}-1)\dd t/t\le m^2e^{m^2T}\dd t$ to obtain
\eqref{eq:small-coefficient}. For \eqref{eq:discrete-small-coefficient}
use $(a_\delta^n-1)/n\le m^2\delta^2e^{m^2n\delta^2}$ and
\[
 \delta^2\sum_{n\le T/\delta^2}(n\delta^2)^{-\alpha/2}
 \le C_\alpha T^{1-\alpha/2}.
\]
Loops with a bounded number of steps are covered at the mesh scale $\delta$;
there are no nonconstant loops below $2\delta^2$.
\end{proof}
Since the coefficient of the covering content tends to zero with
$T$, a single exponential moment of the content controls any fixed
power of the weight.

\subsection{Long loops and the spectral cutoff}\label{sec:long-loops}

We first show that changing the cutoff multiplies the weight by a
deterministic factor and that positivity of the weight is decided by
the slit-domain eigenvalue; we then control long loops in the limit.

\par\nopagebreak[4]
\begin{proposition}
\label{prop:spectral-support}
For compact $K$ with $M_\alpha(K)<\infty$ for some $\alpha<2$,
$I_{m,T}(K)<\infty$ for every $m\ge0$ and finite $T>0$. For $0<T_1<T_2$,
\eqref{eq:cutoff-change} holds.
For $m>0$,
\begin{equation}\label{eq:positivity-exact}
 Y_{m,T}(K)>0\quad\Longleftrightarrow\quad m^2<\lambda_1(D\setminus K).
\end{equation}
\end{proposition}
\begin{proof}
Short-loop finiteness follows from \eqref{eq:small-coefficient}; any
remaining finite duration interval has finite loop mass in $D$. Moving the
cutoff from $T_1$ to $T_2$ adds the hitting mass to $I$ and removes the
avoiding mass from $J$. Their sum is the finite integral over loops in $D$ in
\eqref{eq:cutoff-change}. If $J$ is infinite, both weights are zero, so the
same identity holds without subtracting infinities.
For $U=D\setminus K\ne\varnothing$, the spectral expansion gives
\[
 e^{-t\lambda_1(U)}\le H_U(t)
 \le H_U(1)e^{-(t-1)\lambda_1(U)},\qquad t\ge1.
\]
Thus $J_{m,T}$ is finite exactly when $m^2<\lambda_1(U)$; at equality
the ground state contributes the divergent integral
$\int^\infty(1-e^{-m^2t})\dd t/t$, and above equality the divergence
is exponential. The empty-domain convention completes the proof.
\end{proof}

The next lemma controls the long-loop tail and proves convergence
of the weights, including at the spectral threshold.

\par\nopagebreak[4]
\begin{lemma}\label{lem:IR}
Suppose \eqref{eq:window} holds and
$Q_\delta:=|\gamma^\delta|\log a_\delta+I_{\delta,m,T}(K^\delta)$
converges in probability to the finite variable $I_{m,T}(K)$. Then
$Y_{\delta,m,T}\to Y_{m,T}$ in probability, including on
$\{\lambda_1(D\setminus K)=m^2\}$.
\end{lemma}
\begin{proof}
The case $m=0$ is immediate. Put $U=D\setminus K$,
$U_\delta=D_\delta\setminus K^\delta$,
$p_\delta=p_{\delta,U_\delta}$ and $\rho_\delta=\rho(p_\delta)$.
We first show that finite-window convergence forces a discrete spectral bound on $E_\kappa=\{\lambda_1(U)\ge m^2+4\kappa\}$.
More precisely, for $\kappa>0$ small enough,
\[
 \int_S^{S+1}H_U(t)\frac{\dd t}{t}
 \le\frac{C_D}{S}e^{-(m^2+4\kappa)(S-1)},\qquad S\ge2.
\]
If $\rho_\delta\ge e^{-(m^2+2\kappa)\delta^2}$, then
\[
 \sum_{S\le n\delta^2\le S+1}\frac{\Tr(p_\delta^n)}n
 \ge\frac{c}{S+1}e^{-(m^2+2\kappa)(S+1)}.
\]
Indeed, for even $n$, symmetry gives $\Tr(p_\delta^n)\ge\rho_\delta^n$;
the window contains order $\delta^{-2}$ even integers, each with
$1/n\asymp\delta^2/(S+1)$. Choose a fixed $S$ making the first bound
less than half the second. The intensities then differ by a fixed
positive amount on the exceptional event, so \eqref{eq:window} implies
\begin{equation}\label{eq:gap-prob}
 \PP(E_\kappa\cap\{\rho_\delta\ge
                  e^{-(m^2+2\kappa)\delta^2}\})\longrightarrow0.
\end{equation}
On $E_\kappa$ outside this event, choose even $n_0$ with
$n_0\delta^2\in[1,1+2\delta^2]$. The free heat-kernel bound gives
$\Tr(p_\delta^{n_0})\le C_D$, and for even $n\ge n_0$,
\[
 \Tr(p_\delta^n)\le\rho_\delta^{n-n_0}\Tr(p_\delta^{n_0})
 \le C_{D,m,\kappa}e^{-(m^2+2\kappa)n\delta^2}.
\]
Since odd traces vanish and $a_\delta^n\le e^{m^2n\delta^2}$, the
boosted tail after $R\ge2$ is at most
\[
 C\sum_{n\delta^2>R}\frac{e^{-2\kappa n\delta^2}}n
 \le C\int_{R/2}^\infty e^{-\kappa t}\frac{\dd t}{t}.
\]
The continuum tail has the analogous bound.

For $R>T$ define the duration truncations
\begin{align}
 J_{\delta,m,T,R}&=\sum_{T<n\delta^2\le R}
             \frac{a_\delta^n-1}{n}\Tr(p_\delta^n),\notag\\
 J_{m,T,R}&=\int_T^R(e^{m^2t}-1)H_U(t)\frac{\dd t}{t}.
 \label{eq:lt-long-truncations}
\end{align}
At fixed $R$ these converge in probability by
\eqref{eq:window} and \eqref{eq:lt-window-multiplier}. The tail bounds
and \eqref{eq:gap-prob}, first as $\delta\downarrow0$ then
$R\uparrow\infty$, give convergence of the full $J$ on $E_\kappa$.
Together with $Q_\delta\to I_{m,T}$ this proves convergence of the
weights on $E_\kappa$; letting $\kappa\downarrow0$ covers
$\{\lambda_1(U)>m^2\}$.

On $E_0=\{\lambda_1(U)\le m^2\}$, the continuum truncations tend to
infinity, whereas $J_{\delta,m,T}\ge J_{\delta,m,T,R}$. For
$0<\varepsilon<1$, put $h=\log(1/\varepsilon)$. Then
\begin{align}
 \PP(E_0\cap\{Y_{\delta,m,T}>\varepsilon\})
 &\le\PP(Q_\delta>M)
   +\PP(E_0\cap\{J_{m,T,R}\le M+h+1\})\notag\\
 &\quad+\PP(|J_{\delta,m,T,R}-J_{m,T,R}|>1).
 \label{eq:lt-zero-weight-bound}
\end{align}
After $\delta\downarrow0$, tightness makes the first term small by
choosing $M$ large, and monotone divergence makes the second small
by then choosing $R$ large. This includes spectral equality.
\end{proof}

\subsection{Uniform integrability and convergence of the densities}
\label{sec:ui-laws}

We now combine the three duration regimes and prove uniform
integrability, which gives convergence of the boosted laws.

\begin{proof}[Proof of Theorems~\ref{thm:extension-intro} and
\ref{thm:all-boosts}]
Use the covering and positive-support hypotheses of Theorem~\ref{thm:extension-intro};
in the radial models they follow from Theorem~\ref{thm:covering-input}
and Proposition~\ref{prop:all-support}. Couple
$\eta^\delta\to\eta$ almost surely, with traces $K^\delta,K$, and
abbreviate $M_{\alpha,\delta}(K^\delta)$ to $M_{\alpha,\delta}$.
Equation~\eqref{eq:EB} gives tightness and bounded first moments of
these contents. Self-avoidance bounds the number of vertices in each
mesh-scale box, giving
\[
 |\gamma^\delta|\log a_\delta
 \le Cm^2\delta^{2-\alpha}M_{\alpha,\delta}\longrightarrow0
 \quad\text{in probability}.
\]
At each fixed scale, Hausdorff convergence covers $K$ by fixed
enlargements of the boxes meeting $K^\delta$. Applying this to finitely
many dyadic scales and then increasing their number yields
\begin{equation}\label{eq:lt-limiting-content}
 M_\alpha(K)\le C_\alpha\liminf_{\delta\downarrow0}
                  M_{\alpha,\delta}\quad\text{almost surely}.
\end{equation}
Fatou's lemma gives $\EE e^{(c/C_\alpha)M_\alpha(K)}<\infty$.

For $0<u<T$, split the short-loop difference as
\begin{align}
 |I_{\delta,m,T}-I_{m,T}|
 \le I_{\delta,m,u}+I_{m,u}
 +|(I_{\delta,m,T}-I_{\delta,m,u})
                  -(I_{m,T}-I_{m,u})|.
 \label{eq:lt-short-limit-split}
\end{align}
The first two terms have expectation at most $C_mu^{1-\alpha/2}$ by
\eqref{eq:small-coefficient}--\eqref{eq:discrete-small-coefficient};
the fixed-window difference vanishes by
Lemma~\ref{lem:finite-duration-loops}. Taking $\delta\downarrow0$
then $u\downarrow0$ proves $I_{\delta,m,T}\to I_{m,T}$ in probability.
Lemma~\ref{lem:IR} therefore gives convergence of the weights, including
on the event $\{Y_{m,T}=0\}$.

Fix $p>1$. Choose $0<T\le1$ so that
$pC_{d,\alpha}m^2e^{m^2T}T^{1-\alpha/2}<c/2$, and then choose
$\delta_0$ so that $pCm^2\delta_0^{2-\alpha}<c/2$.
Dropping the subtracted nonnegative long-loop term gives
\[
 Y_{\delta,m,T}^p
 \le\exp\{p[Cm^2\delta^{2-\alpha}
       +C_{d,\alpha}m^2e^{m^2T}T^{1-\alpha/2}]M_{\alpha,\delta}\}
 \le e^{cM_{\alpha,\delta}}.
\]
Consequently
\begin{equation}\label{eq:lt-uniform-p-moment}
 \sup_{0<\delta<\delta_0}\EE Y_{\delta,m,T}^p<\infty,
\end{equation}
and convergence is in $L^1$.
The positive-support hypothesis and \eqref{eq:positivity-exact} give
\begin{equation}\label{eq:lt-normalizers}
 Z_{\delta,m,T}:=\EE Y_{\delta,m,T}
 \longrightarrow Z_{m,T}:=\EE Y_{m,T}\in(0,\infty).
\end{equation}
For $m=0$ this is immediate. Lemma~\ref{lem:discrete-comp} identifies
$Y_{\delta,m,T}$ as a positive deterministic multiple of
$W_{\delta,m}$, so the extended discrete laws are defined for all
small meshes.

For bounded continuous $F$ on the naturally parametrized curve space,
\begin{align*}
 |\EE[F(\eta^\delta)Y_{\delta,m,T}]-\EE[F(\eta)Y_{m,T}]|\le\|F\|_\infty\EE|Y_{\delta,m,T}-Y_{m,T}|
       +\EE[|F(\eta^\delta)-F(\eta)|Y_{m,T}]\longrightarrow0.
\end{align*}
The first term uses $L^1$ convergence; the second uses dominated
convergence under the curve coupling. Dividing by the normalizers and using
\eqref{eq:lt-normalizers} gives
\[
 \EE_{\PP_\delta^{(m)}}F
 =\frac{\EE[F(\eta^\delta)Y_{\delta,m,T}]}{Z_{\delta,m,T}}
 \longrightarrow\frac{\EE[F(\eta)Y_{m,T}]}{Z_{m,T}}
 =\EE_{\PP^{(m)}}F.
\]
The cutoff identity \eqref{eq:cutoff-change} multiplies the density and its
mean by the same deterministic factor. Its discrete counterpart and
convergence for loops in $D$ on fixed duration windows also extend the $L^1$ result to every
fixed $T>0$.
Finally, when $m^2<\lambda_1(D)$,
Lemma~\ref{lem:subcritical-exit-moments} gives
$\EE^0a_\delta^{T_\delta}<\infty$ for small $\delta$, and
Proposition~\ref{prop:agreement} identifies the extended law with the
loop erasure of the normalized boosted walk.
\end{proof}

\section{Absolute continuity and geometric consequences}
\label{sec:spectral-geometry}

We now identify the subcritical density, determine the regimes of
absolute continuity, and bound the size of unvisited regions.

\subsection{The subcritical density}

Below the spectral threshold of $D$, restoring the finite loop
factor of $D$ gives the density in Theorem~\ref{thm:subcritical}. Its
normalizer is identified by the following exit-time estimate.

\par\nopagebreak[4]
\begin{lemma}
\label{lem:subcritical-exit-moments}\label{lem:exit-moments}
Let $T_D$ be the exit time of Brownian motion with generator
$\Delta/(2d)$, and $T_\delta$ the lattice exit time. For every
$a<\lambda_1(D)$ there is $\delta_0(a)>0$ such that
\[
 \sup_{0<\delta\leq\delta_0(a)}\sup_{x\in D_\delta}
       \EE^x e^{a\delta^2T_\delta}<\infty,
 \qquad
 \EE^0e^{a\delta^2T_\delta}\longrightarrow\EE^0e^{aT_D}.
\]
Here $\delta^2T_\delta$ is the raw walk's diffusive lifetime.
\end{lemma}

\begin{proof}
Fix $\max\{a,0\}<c<c'<\lambda_1(D)$. The Brownian semigroup gives
$\sup_x\PP^x(T_D>T)<e^{-c'T}$ for a sufficiently large fixed $T$.
The exit-time invariance principle transfers this to
$\sup_x\PP^x(T_\delta>\lfloor T/\delta^2\rfloor)\le e^{-cT}$
for small $\delta$, uniformly in the starting point.
Otherwise, there would be meshes \(\delta_n\downarrow0\) and starting points \(x_n\in D_{\delta_n}\) violating the bound. Passing to a subsequence, \(x_n\) converges in \(\overline D\);
interior convergence or boundary regularity of the sphere gives a
contradiction. Iterating by the Markov property yields
\begin{equation}\label{eq:ss-iterated-survival-block}
 \sup_{x\in D_\delta}\PP^x(\delta^2T_\delta>t)
 \le C_c e^{-ct},\qquad t\ge0.
\end{equation}
Integration gives the moment bound. Using an exponent between
$\max\{a,0\}$ and $c$ gives uniform integrability, so convergence
of the exit times implies convergence of their exponential moments.
\end{proof}

\begin{proof}[Proof of Theorem~\ref{thm:subcritical}]
Fix $m^2<\lambda_1(D)$. Short-loop control and the spectral gap in $D$
give $\Phi_m(K)<\infty$ almost surely and
\[
 \log Y_{m,T}(K)=\Phi_m(K)-
 \int_T^\infty(e^{m^2t}-1)H_D(t)\frac{\dd t}{t}.
\]
Lemma~\ref{lem:exit-moments} makes the boosted walk normalizable for
small $\delta$, so Proposition~\ref{prop:agreement} and the compensated
lattice identity give
\[
 w_{\delta,m}=\EE^0[a_\delta^{T_\delta}\mid\gamma^\delta]
 =Y_{\delta,m,T}\exp\left\{
 \sum_{n\delta^2>T}\frac{a_\delta^n-1}{n}
 \Tr(p_{\delta,D_\delta}^n)\right\}.
\]
The deterministic sum converges to the Brownian integral. Finite
windows follow from Lemma~\ref{lem:finite-duration-loops}; for the tail,
choose $m^2<c<\lambda_1(D)$ and use
$\Tr(p_{\delta,D_\delta}^n)\leq Ce^{-cn\delta^2}$ for $n\delta^2\geq2$.
To obtain this bound, split the return kernel at a time of order
$\delta^{-2}$ and apply the free heat-kernel bound to one factor and
the survival estimate
\eqref{eq:ss-iterated-survival-block}, uniform in the starting vertex,
to the other. The Brownian tail
follows from the killed semigroup. Hence Section~\ref{sec:transfer}
gives $w_{\delta,m}\to e^{\Phi_m(K)}$ in probability under its curve coupling.

For $p>1$ with $pm^2<\lambda_1(D)$, conditional Jensen gives
\[
 \EE_{\PP_\delta}w_{\delta,m}^p
 \leq\EE^0a_\delta^{pT_\delta}
 \leq\EE^0e^{pm^2\delta^2T_\delta}\leq C_{p,m}.
\]
Since $\nu_\delta=\delta^{-2}\log a_\delta\to m^2$, exit-time
convergence and this uniform integrability also give
\[
 \EE_{\PP}e^{\Phi_m(K)}
 =\lim_{\delta\downarrow0}\EE^0a_\delta^{T_\delta}
 =\EE^0e^{m^2T_D}.
\]
This proves \eqref{eq:subcritical-density}. Fatou gives the $L^p$
assertion, and strict positivity gives mutual absolute continuity.
\end{proof}

\subsection{Spectral support and absolute continuity regimes}

We identify the essential infimum of the slit-domain eigenvalue and show that the eigenvalue lies strictly above it almost surely in the
following proposition, proved in Appendix~\ref{app:spectral}.

\par\nopagebreak[4]
\begin{proposition}
\label{prop:thresholds}
Let $X=\lambda_1(D\setminus K)$. In dimension three,
$\essinf X=\lambda_1(D)$ and $\PP(X>\lambda_1(D))=1$.

In the standard planar radial setting,
$\Lambda_2:=\essinf X=\pi^2/4$ and $\PP(X>\pi^2/4)=1$.

\end{proposition}

\begin{proof}[Proof of Corollary~\ref{cor:phase-space}]
The density formula gives $\PP^{(m)}\ll\PP$ and, by
Proposition~\ref{prop:spectral-support},
\[
 \PP\ll\PP^{(m)}
 \quad\Longleftrightarrow\quad\PP\{Y_{m,T}>0\}=1
 \quad\Longleftrightarrow\quad\PP(X>m^2)=1.
\]

Proposition~\ref{prop:thresholds} gives both equivalences, including
their endpoints.
\end{proof}

\subsection{Unvisited regions at large boost}

\begin{proof}[Proof of Corollary~\ref{cor:space-filling}]
Theorem~\ref{thm:all-boosts} gives
$m^2<\lambda_1(D\setminus K)$, $\PP^{(m)}$-almost surely.
For any open cube $Q\subset D\setminus K$, domain monotonicity and
the eigenvalue of a cube for $-\Delta/(2d)$ give
\[
 m^2<\lambda_1(D\setminus K)\le\lambda_1(Q)
     =\frac{\pi^2}{2\operatorname{side}(Q)^2}.
\]
This proves~\eqref{eq:empty-cube}, simultaneously for all such cubes.
If $m>M(\ell)$, the bound is smaller than $\ell$, which proves the
second assertion.
\end{proof}

\section{The chordal driving function in two dimensions}
\label{sec:chordal}

In this section, we construct a continuum chordal law using the loop-decoration
theorem of Berestycki--Sauzedde~\cite{BS26}, and then identify the drift
of its driving function throughout the subcritical range $q<\lambda_1(D)$.
As discussed in Section~\ref{sec:intro}, the calculation follows the
martingale strategy of Makarov--Smirnov~\cite{MS} and
Chelkak--Wan~\cite{CW}; its additional input is a Green-kernel
comparison uniform over the evolving slits.

\subsection{Continuum tilt and the driving statement}
\label{subsec:chordal-setting}
\label{subsec:chordal-capacity}

Let $D$ be a bounded simply connected planar domain with analytic Jordan
boundary, $a,b\in\partial D$ distinct, and $\mathbb P_{D;a,b}$ the law
of chordal ${\rm SLE}_2$, with curve $\eta$ and trace $K$.  Brownian
motion has generator $L=\Delta/4$; write $A_U=-\Delta_U/4$ and
$G_U=A_U^{-1}$, also for its kernel.  Set $q=m^2<\lambda_1(D)$ and
\begin{equation}\label{eq:chordal-loop-functional}
 V_m(K)=\int_{\ell\subset D}(e^{qt_\ell}-1)
             \mathbf1_{\{\ell\cap K\ne\varnothing\}}\,\mu_D(d\ell).
\end{equation}
Here $t_\ell$ is Brownian duration.  For capacity time, fix
$\phi:D\to\mathbb H$ by $\phi(a)=0$, $\phi(b)=\infty$ and
$\Im\left(\phi(o)\right)=1$ at a reference point $o\in D$, and use
\begin{equation}\label{eq:chordal-loewner}
 g_t(z)=z+\frac{2t}{z}+O(|z|^{-2}),\qquad
 \partial_tg_t(z)=\frac{2}{g_t(z)-\xi_t}.
\end{equation}
Thus $\operatorname{hcap}(\phi(\eta[0,t]))=2t$.  Put
$D_t=D\setminus\eta[0,t]$, $a_t=\eta(t)$ and
$f_t=(g_t-\xi_t)\circ\phi$.  This clock differs from Brownian duration
and from the natural parametrization used in the radial theorem.

The normalizer is an exponential moment of the duration of a
Brownian excursion.

\par\nopagebreak[4]
\begin{proposition}
\label{prop:chordal-tilt}
For $q<\lambda_1(D)$,
\begin{equation}\label{eq:chordal-normalizer}
 Z_m(D;a,b):=\mathbb E_{D;a,b}e^{V_m(K)}
       =\mathbb E^{\mathrm{exc}}_{D;a,b}e^{qT}<\infty,
\end{equation}
where $T$ is the Brownian excursion duration.  The law
\begin{equation}\label{eq:chordal-density}
 \frac{d\mathbb P^{(m)}_{D;a,b}}{d\mathbb P_{D;a,b}}(K)
       =\frac{e^{V_m(K)}}{Z_m(D;a,b)}
\end{equation}
is equivalent to ${\rm SLE}_2$, with density in $L^p$ whenever
$p>1$ and $pq<\lambda_1(D)$.
\end{proposition}

\begin{proof}

The chordal extension of \cite[Theorem~1.1, Section~1.1]{BS26}
attaches an independent intensity-one loop soup to ${\rm SLE}_2$
in $(\mathbb H;0,\infty)$, producing a Brownian excursion, with
endpoint localization included. Their rooted loop measure
\cite[Section~1.2]{BS26} is
$\dd x\,\dd t/(2\pi t^2)$ times the probability bridge for $\Delta/2$.
Map to $D$ with the conformal clock and then double durations:
this becomes $\dd x\,\dd t/(\pi t^2)$ times the bridge for $\Delta/4$,
exactly our $\mu_D$. Thus no extra factor of two enters the soup
intensity, and the excursion duration $T$ is the sum of its attached
loop durations.
  The Poisson exponential formula, first on
finite duration windows and then by monotone convergence, gives
$\mathbb E[e^{qT}\mid K]=e^{V_m(K)}$.

The excursion duration density is proportional to the inward normal
derivative at both endpoints of the killed heat kernel.  At large
times it is bounded by $Ce^{-\lambda_1(D)t}$: differentiate the outer
factors in a three-factor semigroup decomposition and apply the
$L^2$ spectral bound to the middle factor.  Analytic boundary makes
the differentiated outer factors $L^2$.  At small times the distinct
endpoints give an integrable Gaussian off-diagonal bound.  Hence
\eqref{eq:chordal-normalizer} is finite.  Conditional Jensen gives
$\mathbb E e^{pV_m(K)}\le\mathbb E^{\mathrm{exc}}e^{pqT}<\infty$;
the density is therefore finite and strictly positive almost surely.
\end{proof}

Assume that $D\subset B(0,R)$. With the absolute constant $C_0$ of
Lemma~\ref{lem:chordal-gauge}, write
\begin{equation}\label{eq:chordal-subcritical}
\rho:=\frac{q}{\lambda_1(D)}<1,\qquad M_R=C_0R^2.
\end{equation}
Throughout the subcritical range $q<\lambda_1(D)$, the drift of the driving function is a ratio
of boosted Poisson kernels.

\par\nopagebreak[4]
\begin{theorem}
\label{thm:chordal-driving}
For $q<\lambda_1(D)$, the law~\eqref{eq:chordal-density} has a
continuous driving semimartingale with quadratic variation $2t$.
Its driving function satisfies
\begin{equation}\label{eq:chordal-driving}
 d\xi_t=\sqrt2\,dB_t^{(m)}+2\lambda_t\,dt,
 \qquad \lambda_t=\frac{Q_t^{(m)}(b)}{P_t^{(m)}(b)},
\end{equation}
where $B^{(m)}$ is real standard Brownian motion and the quotient is
the boundary limit in~\eqref{eq:chordal-lambda}.
\end{theorem}

The semimartingale assertion follows from
Proposition~\ref{prop:chordal-tilt} and Girsanov's theorem, as in
\cite[Section~2.6]{CW}. Below we identify the drift using
\cite[Sections~4.1--4.3]{CW}. To translate their conventions, replace
their operator $-\Delta+\mu^2$ by $-\Delta-4q$ and multiply its
Green kernel by four. The massive upper bounds are replaced by the
uniform boosted comparison in Lemma~\ref{lem:chordal-gauge}.

\subsection{Boosted Green and Poisson kernels}
\label{subsec:chordal-kernels}
We now introduce the boosted kernels and establish comparisons that are uniform over all slit domains inside $D$.

For $U\subset D$, domain monotonicity gives $q<\lambda_1(U)$, and
\begin{equation}\label{eq:chordal-resolvent}
 G_U^{(m)}=(A_U-q)^{-1}=G_U+qG_UG_U^{(m)}.
\end{equation}
All kernels are extended by zero outside their domains.

The key estimate compares boosted and ordinary kernels, with a
constant that does not depend on the shape of the domain.

\par\nopagebreak[4]
\begin{lemma}
\label{lem:chordal-gauge}
For every bounded planar open set $U\subset B(0,R)$ and $x,y$ in
one connected component,
\begin{equation}\label{eq:chordal-3g}
 \int_U G_U(x,w)G_U(w,y)dw\le M_RG_U(x,y),\qquad M_R=C_0R^2.
\end{equation}
If, in addition, $U\subset D$ is simply connected and
$q<\lambda_1(D)$, then
\begin{equation}\label{eq:chordal-green-comparison}
 G_U\le G_U^{(m)}\le C_mG_U,\qquad C_m=C_m(D,q)<\infty,
\end{equation}
where $C_m$ is independent of $U$.
The same comparison holds for Poisson and excursion kernels obtained
by Martin-boundary limits at either or both endpoints.
\end{lemma}

\begin{proof}
Bass--Burdzy
\cite[Theorem~1.1]{BB95} bounds
$G_U(x,w)G_U(w,y)/G_U(x,y)$ by a diameter-dependent constant times
$1+\log^+(1/|x-w|)+\log^+(1/|w-y|)$ (which is known as the planar $3G$ inequality).
Rescale to diameter at most two and integrate to obtain
\eqref{eq:chordal-3g}; the change from their generator $\Delta/2$
to $\Delta/4$ only changes $C_0$.

For simply connected $U\subset D$, $G_U$ is quasi-metrically
modifiable with an absolute constant. Indeed, choose a conformal map
$\psi:U\to\mathbb H$ and put $h=\Im(\psi)$. The reciprocal of
$G_U(x,y)/(h(x)h(y))$, in half-plane coordinates and up to a constant,
is
\[
 d(z,w)=\frac{\Im (z)\,\Im (w)}
 {\log\bigl(1+4\Im (z)\,\Im \left(\frac{w}{|z-w|^2}\right)\bigr)}.
\]
To check the quasi-triangle inequality, put $r=|z-w|$ and
$H=\max\{\Im (z),\Im (w)\}$. Always $d(z,w)\ge r^2/4$, and
$d(z,w)\asymp r^2$ when $r\ge H/2$, so this case follows from the
Euclidean triangle inequality. If $r<H/2$, then $d(z,w)\le CH^2$.
A third point at distance at least $H/4$ from either endpoint is
therefore handled by the lower bound. Otherwise all three heights
are comparable, and one of the two joining distances is at least
$r/2$; the displayed formula, or monotonicity of
$a/\log(1+4a/s^2)$ in $a,s>0$, gives the required inequality with
an absolute constant.

For $q>0$, we use the exponential upper bound for kernels
of Neumann series from Frazier--Nazarov--Verbitsky
\cite[Corollary~3.3]{FNV}, which holds for quasi-metrically
modifiable kernels when the associated operator has norm less
than one. Apply this with kernel $G_U$ and measure $q\,dx$.
The associated operator has $L^2(q\,dx)$ norm
$q/\lambda_1(U)\le\rho<1$, so their estimate gives
\[
 G_U^{(m)}(x,y)\le G_U(x,y)
 \exp\!\left\{C_\rho q\,
 \frac{\int_U G_U(x,w)G_U(w,y)\,dw}{G_U(x,y)}\right\}
 \le e^{C_\rho qM_R}G_U(x,y).
\]
Here $C_\rho$ is independent of $U$. The lower bound follows from
positivity of the Neumann series, and $q=0$ is immediate.

For the boundary statement, fix a simply connected $U$
and two distinct prime ends $a,b$. Choose
$\psi:U\to\mathbb H$ with $\psi(a)=0$ and $\psi(b)=\infty$,
and put $P(w)=-\frac1\pi\Im\left(\psi(w)^{-1}\right)$, $k(w)=4\Im\left(\psi(w)\right)$.
Then we have $\lim_{y\to b}\frac{G_U(x,y)}{P(y)}=k(x)$ and $\lim_{x\to a}\frac{G_U(x,y)}{k(x)}=P(y)$.
Write $G_U^{[j]}$ for the $j$-th operator power, with
$G_U^{[0]}=I$, and define
\[
 F_j(x,y)=\frac{G_U^{[j+1]}(x,y)}{G_U(x,y)},
 \qquad
 B_{x,y}(w)=
 \frac{G_U(x,w)G_U(w,y)}{G_U(x,y)}.
\]
The $3G$ inequality used above gives $0\le F_j\le M_R^j$ and $\sup_{x\ne y\in U}\int_U B_{x,y}(w)^2\,dw
 \le C_R$.
Indeed, the logarithmic majorant used above has uniformly
bounded $L^2$ norm on $B(0,R)$. Moreover, for fixed $w$,
\[
 B_{x,y}(w)\longrightarrow
 \begin{cases}
  G_U(x,w)k(w)/k(x),&y\to b,\quad x\text{ fixed},\\
  P(w)k(w),&x\to a,\quad y\to b.
 \end{cases}
\]
Using the recursion $F_j(x,y)=\int_U B_{x,y}(w)F_{j-1}(w,y)\,dw$ and $F_0=1$,
induction and uniform integrability therefore give, for every
fixed $j\ge0$,
\[
 \lim_{y\to b}
 \frac{G_U^{[j+1]}(x,y)}{P(y)}
 =(G_U^{[j]}k)(x),
\]
and, by symmetry, the analogous limit at $a$ with $P$ and
$k$ interchanged. Applying the same recursion at both endpoints
gives, for $j\ge1$,
\[
 \lim_{\substack{x\to a\\y\to b}}
 \frac{G_U^{[j+1]}(x,y)}{G_U(x,y)}
 =\int_U P(w)(G_U^{[j-1]}k)(w)\,dw.
\]
All these integrals are finite by the same $3G$ inequality.

It remains to interchange the boundary limits with the
Neumann series. Choose $q<q'<\lambda_1(D)$. Positivity and
the comparison already proved at $q'$ yield, uniformly in
$x,y$ and $U$,
\[
 \begin{aligned}
 \frac{\sum_{j\ge n}q^jG_U^{[j+1]}(x,y)}{G_U(x,y)}
 \le
 \left(\frac q{q'}\right)^n
 \frac{G_U^{(\sqrt{q'})}(x,y)}{G_U(x,y)}\le C(D,q')\left(\frac q{q'}\right)^n,
 \qquad n\ge1.
 \end{aligned}
\]
Thus the normalized series converges uniformly, and its
boundary limits can be taken term by term. The one-sided
limits give the boosted Poisson kernels, while the two-sided
limit gives the boosted excursion kernel. Passing to the
limits in the interior comparison proves the asserted
boundary comparisons with the same constant $C_m$.
\end{proof}

Define the chart-dependent kernels
\begin{equation}\label{eq:chordal-pq}
 P_t(z)=-\pi^{-1}\Im \left(f_t(z)^{-1}\right),\qquad
 Q_t(z)=-\pi^{-1}\Im \left(f_t(z)^{-2}\right),
\end{equation}
and, for $H=P,Q$,
\begin{equation}\label{eq:chordal-boosted-p}
 H_t^{(m)}(z)=H_t(z)+q\int_{D_t}H_t(w)G_{D_t}^{(m)}(w,z)dw.
\end{equation}
The estimates below justify the $Q$ integral for almost every time.
The potential is constant in physical coordinates: a conformal map
$F:U\to\mathbb H$ changes it to $q|(F^{-1})'|^2$.
At the target put
\begin{equation}\label{eq:chordal-target-kernel}
 k_t^{(m)}(w)=\lim_{z\to b}\frac{G_{D_t}^{(m)}(w,z)}{P_t(z)},
 \qquad R_t=\lim_{z\to b}\frac{P_t^{(m)}(z)}{P_t(z)}.
\end{equation}
Our Green normalization gives $k_t(w)=4\Im \left(f_t(w)\right)$.
Lemma~\ref{lem:chordal-gauge} and the resolvent identity yield
\begin{gather}
 k_t\le k_t^{(m)}\le C_mk_t,\qquad
 P_t\le P_t^{(m)}\le C_mP_t,\qquad 1\le R_t\le C_m,
 \label{eq:chordal-boundary-comparison}\\
 R_t=1+q\int_{D_t}P_t(w)k_t^{(m)}(w)dw.
 \label{eq:chordal-r-integral}
\end{gather}
We shall repeatedly use the direct half-plane bounds
\begin{equation}\label{eq:chordal-pq-k}
 P_tk_t\le4/\pi,\qquad |Q_t|k_t\le4P_t.
\end{equation}

\subsection{Conditional density and domain variation}
\label{subsec:chordal-observable}

We express the conditional density through Poisson kernels and
compute how these kernels change as the slit grows.

Conditioning on an initial segment separates the loop weight
accumulated so far from the partition function of the remaining
domain.

\par\nopagebreak[4]
\begin{lemma}
\label{lem:chordal-conditional-density}
For $q<\lambda_1(D)$, let $F_t=V_m(\eta[0,t])$.
Then $F_t$ is continuous and increasing, and
\begin{equation}\label{eq:chordal-density-factorization}
 \mathcal D_t=\mathbb E_{D;a,b}
      [e^{V_m(K)}/Z_m\mid\mathcal F_t]=e^{F_t}R_t/R_0.
\end{equation}
\end{lemma}

\begin{proof}
Split loops hitting $K$ into those hitting the initial segment and
those contained in $D_t$ hitting the remainder.  Loop restriction
and the ${\rm SLE}_2$ domain Markov property identify the latter
partition function with that for $(D_t;a_t,b)$.  By loop decoration
and Feynman--Kac it is the excursion duration moment $R_t$.
The conformal clock is as in Proposition~\ref{prop:chordal-tilt};
Lemma~\ref{lem:chordal-gauge} justifies the boundary passage.
Finite-time slits have locally connected boundary, tip prime end
$a_t$, and an analytic target neighborhood.  At $t=0$ this gives
$Z_m=R_0$, and hence~\eqref{eq:chordal-density-factorization}.

The weighted loop measure hitting $K$ is finite almost surely.
Its pushforward by first intersection time has no atom at $s$:
such loops would pass through $\eta(s)$, and points are polar for
planar Brownian motion.  Conditionally on the curve this applies to
every $s$.  Thus $F_t$ is continuous with finite variation.
\end{proof}

The variation of the kernels is given by Hadamard's formula.

\par\nopagebreak[4]
\begin{lemma}
\label{lem:chordal-hadamard}
For $q<\lambda_1(D)$, while the points remain in the slit,
with $w\ne z$ in the first identity,
\begin{align}
 G_{D_s}^{(m)}(w,z)-G_{D_t}^{(m)}(w,z)
    &=8\pi\int_s^tP_u^{(m)}(w)P_u^{(m)}(z)du,
    \label{eq:chordal-hadamard}\\
 dk_t^{(m)}(w)&=-8\pi P_t^{(m)}(w)R_tdt.
    \label{eq:chordal-target-variation}
\end{align}
Under ordinary ${\rm SLE}_2$, almost surely,
\begin{equation}\label{eq:chordal-integrability}
 \int_0^\infty\left(\int_{D_t}P_t\right)^2dt<\infty,
 \qquad \int_0^\infty\int_{D_t}P_t^2\,dw\,dt<\infty.
\end{equation}
The first integral has a deterministic bound depending on $D$ and
its chart. On the diagonal, \eqref{eq:chordal-hadamard} means the
difference of finite parts, obtained by adding the common singularity
$(2/\pi)\log|w-z|$ to each Green kernel before letting $z\to w$.
\end{lemma}

\begin{proof}
We adapt the argument of~\cite[Lemma~4.7]{CW}, replacing
their increment domination by the positive Neumann series.
Extend all kernels by zero to $D$ and write $G_t^{[0]}=I$.
The ordinary Hadamard formula, in our normalization, gives $ G_s(w,z)-G_t(w,z)
 =8\pi\int_s^t P_u(w)P_u(z)\,du$.
Applying the product rule to finite convolution powers yields
\[
 \begin{aligned}
 G_s^{[n+1]}(w,z)-G_t^{[n+1]}(w,z)=8\pi\int_s^t\sum_{j=0}^n
 (G_u^{[j]}P_u)(w)
 (G_u^{[n-j]}P_u)(z)\,du.
 \end{aligned}
\]
All terms are nonnegative, so Tonelli's theorem permits
summation against $q^n$. The uniform Green comparison
\eqref{eq:chordal-green-comparison} ensures convergence,
while the resolvent expansion identifies
$\sum_{j\ge0}q^jG_u^{[j]}P_u=P_u^{(m)}$.
This proves \eqref{eq:chordal-hadamard}.

For the target variation, divide the integrated identity
by $P_s(z)$ and let $z\to b$. The hydrodynamic normalization
gives $P_u(z)/P_s(z)\to1$, while
$P_u^{(m)}\le C_mP_u$ and the ordinary Hadamard formula
justify passage through the time integral. Hence
\[
 k_s^{(m)}(w)-k_t^{(m)}(w)
 =8\pi\int_s^t P_u^{(m)}(w)R_u\,du,
\]
which gives \eqref{eq:chordal-target-variation}.
The diagonal identity follows by subtracting the common
logarithmic singularity. Finally, the integrability
statements are~\cite[Corollary~4.6]{CW}; the change of
Green normalization replaces their coefficient $2\pi$
by $8\pi$.
\end{proof}

\subsection{Identification of the drift}
\label{subsec:chordal-drift}

\begin{proof}[Proof of the drift formula in Theorem~\ref{thm:chordal-driving}]
Apply~\cite[Lemmas~4.8--4.9 and equation~(4.7)]{CW} with the same
replacements, using~\eqref{eq:chordal-boundary-comparison} and
$|Q_t|k_t^{(m)}\le4C_mP_t$ for domination. Under the ordinary law this
gives, with $I_t=\int_{D_t}P_tP_t^{(m)}\,dw$,
\begin{equation}\label{eq:chordal-dr}
 dR_t=\sqrt2q\left(\int_{D_t}Q_tk_t^{(m)}\,dw\right)dB_t
           -8\pi qR_tI_t\,dt,
\end{equation}
and the boundary identity
\begin{equation}\label{eq:chordal-lambda}
 \frac{Q_t^{(m)}(b)}{P_t^{(m)}(b)}
 :=\lim_{z\to b}\frac{Q_t^{(m)}(z)}{P_t^{(m)}(z)}
 =\frac q{R_t}\int_{D_t}Q_tk_t^{(m)}\,dw=:\lambda_t.
\end{equation}
Since $e^{F_t}R_t/R_0$ is a martingale and $F_t$ has finite variation,
its finite-variation terms cancel. Its stochastic logarithm is
$\sqrt2\lambda_t\,dB_t$, and Girsanov gives~\eqref{eq:chordal-driving}.
\end{proof}

With the domain and chart fixed,
$\left.\partial_x\left[-\pi^{-1}\Im\left([g_t(\phi(z))-x]^{-1}\right)\right]
\right|_{x=\xi_t}=Q_t(z)$.
This explains the source-derivative notation of~\cite{MS,CW};
\eqref{eq:chordal-lambda} is the precise drift used here.

As in~\cite[Remark~4.10]{CW}, the bound
$|\lambda_t|\le4qC_m\int_{D_t}P_t$ and
\eqref{eq:chordal-integrability} imply Novikov's condition and weak uniqueness among
driving laws locally equivalent to that of ordinary
chordal ${\rm SLE}_2$. These conclusions hold throughout
$q<\lambda_1(D)$.

\section{Sharp box-counting estimates}
\label{sec:sharp-box-counting}\label{sec:surviving-bridges}
\label{sec:planar-covering}\label{sec:sharp-local}
\label{sec:three-d-covering}\label{sec:ordered-boxes}\label{sec:geometry}

In this section we prove Theorem~\ref{thm:covering-input}. The argument mimics the moment
method of Barlow--Masson~\cite[Section~5]{BM} in dimension two and of
Shiraishi~\cite[Section~8]{Shiraishi} in dimension three, with points
replaced by boxes. Two steps need a modification: the decomposition
at the ordered entrances into the boxes, where visits that are later
erased must be discarded, and the estimate at a single box, where the
sharp hitting exponents enter. We state these two steps precisely and
only indicate the remaining, standard, arguments.

Throughout this section we work on $\ZZ^d$ with unit mesh. Let
$D_N=B(0,N)\cap\ZZ^d$, let $S$ be simple random walk started at $0$,
and let $\gamma_N=\LE(S[0,T_{D_N}])$, including the exit vertex. For
$1\le L\le N$, let $\mathcal Q_L$ be a fixed translate of the grid of
cubes of side $L$, and put
\[
 M_L=\#\{Q\in\mathcal Q_L:Q\cap\ran(\gamma_N)\ne\varnothing\}.
\]
Recall \eqref{eq:b23} for the definition of $b_d$ and set $\theta_d=2-b_d$.
\begin{theorem}\label{thm:sharp-fixed-scale}\label{thm:planar-cover-exp}
\label{prop:r3-fixed-scale}
There are constants $c,C>0$, depending only on $d$, such that,
uniformly in the translation of the grid,
\begin{equation}\label{eq:sharp-fixed-scale}
 \PP\{M_L>u(N/L)^{b_d}\}\le Ce^{-cu},
 \qquad 1\le L\le N,\quad u\ge1.
\end{equation}
\end{theorem}

Rescaling by $N^{-1}$ gives the fixed-scale bound
\eqref{eq:sharp-fixed-scale-input} with $\delta=N^{-1}$ and $r=L/N$.
Indeed, a cube meeting the interpolated trace lies within distance
$\delta\le r$ of a vertex of $\gamma^\delta$, so $N_r(K^\delta)$ is at
most $3^d$ times the number of cubes containing a vertex of
$\gamma^\delta$.

\subsection{Reduction to ordered box entrances}
\label{sec:ordered-visits}\label{sec:box-outline}

Fix a family $\mathcal F\subset\mathcal Q_L$ of pairwise disjoint
cubes. The cubes of $\mathcal F$ visited by $\gamma_N$ have a
unique order of first entrance. For distinct cubes
$Q_1,\ldots,Q_k\in\mathcal F$, write
$\{Q_1\to\cdots\to Q_k\}$ for the event that $\gamma_N$ hits all
of them and first enters them in this order. If $X$ is the number
of cubes of $\mathcal F$ hit by $\gamma_N$, then
\begin{equation}\label{eq:os-factorial-moment-expansion}
 \EE\binom Xk
 =\sum_{\substack{Q_1,\ldots,Q_k\in\mathcal F\\ \text{distinct}}}
 \PP\{Q_1\to\cdots\to Q_k\}.
\end{equation}
If $\EE\binom Xk\le B^k$ for all $k\ge1$ and some $B\ge1$, then
$\EE(1+(2B)^{-1})^X\le2$, and since $\log(1+t)\ge t/2$ on $[0,1]$,
$\EE e^{X/(4B)}\le2$. It therefore suffices to prove
\begin{equation}\label{eq:box-outline-moments}
 \EE\binom Xk\le[C(N/L)^{b_d}]^k,\qquad k\ge1,
\end{equation}
for each of a bounded number of families $\mathcal F$, which together
cover all cubes except those treated separately below.

The families are chosen as follows. Fix a large constant $A_0$. If
$N/L\le A_0$, the theorem is trivial, since the number of cubes is
bounded. Cubes within distance $A_0L$ of $0$ are $O(1)$ in number.
Cubes within distance $A_0L$ of $\partial B(0,N)$ are counted by
Lemma~\ref{lem:r3-boundary-count} below. The remaining cubes are
split into a bounded number of families $\mathcal F$ such that the
centers of any two distinct cubes in one family are at distance at
least $A_0L$.

\begin{lemma}\label{lem:r3-boundary-count}
Let $M_L^\partial$ be the number of cubes meeting both $\ran(\gamma_N)$
and the $A_0L$-neighborhood of $\partial B(0,N)$. Then
\begin{equation}\label{eq:sharp-boundary-tail}
 \PP(M_L^\partial>u)\le Ce^{-cu},\qquad u\ge1,
\end{equation}
uniformly in $1\le L\le N$ and in the translation of the grid.
\end{lemma}

\begin{proof}
We count the cubes met by $S[0,T_{D_N}]$. Whenever $S$ visits the
$(A_0+2\sqrt d)L$-neighborhood of the sphere outside the balls selected
so far, select the ball of radius $AL$ around its current position,
where $A$ is a large constant. Uniformly in the past, the walk exits
$D_N$ before leaving this ball with probability at least $p>0$, since
the ball contains a ball of radius $cL$ outside $B(0,N)$ within
distance $AL/2$ of its center. By the strong Markov property the
number of selected balls has a geometric tail, and each selected ball
meets at most $C_A$ cubes.
\end{proof}

\subsection{The decomposition at surviving entrances}
\label{sec:exact-expansion}

The first visit of the walk to a cube may be erased by a later loop. We therefore mark the last visit to the first vertex of the cube retained by the final loop erasure,
 that is, the index $t_Q$ at the vertex
$y_Q$ of \eqref{eq:surviving-first-entrance}. For a finite set
$U\subset\ZZ^d$ and $x,y\in U$, the \emph{Green bridge} measure is
\begin{equation}\label{eq:greenbridge}
 \br_U^{x,y}(\omega)=\frac{p(\omega)}{G_U(x,y)}
\end{equation}
on finite nearest-neighbor paths $\omega$ from $x$ to $y$ in $U$; by
\eqref{eq:ordinary-green} it is a probability measure.

\begin{lemma}\label{lem:surviving-entrance}
Let $Q\subset U$, $x\in U\setminus Q$ and $y\in Q$. For a path
$S[0,T_U]$ from $x$ that visits $y$, let $t$ be its last visit to $y$
and $\alpha=\LE(S[0,t])$. Then $y$ is the first vertex of $Q$ on
$\LE(S[0,T_U])$ if and only if
\begin{equation}\label{eq:survival}
 \alpha\cap Q=\{y\},\qquad \alpha\cap S[t+1,T_U]=\varnothing.
\end{equation}
Consequently, if $Z\sim\br_U^{x,y}$ and $\widetilde S$ is an
independent walk from $y$ stopped on leaving $U$, then
\begin{equation}\label{eq:singlemark}
 \PP^x\{y_Q=y\}=G_U(x,y)\,
 \PP\bigl\{\LE(Z)\cap Q=\{y\},\
 \LE(Z)\cap\widetilde S[1,\widetilde T_U]=\varnothing\bigr\}.
\end{equation}
\end{lemma}

\begin{proof}
Write $\alpha=(a_0,\ldots,a_l)$ with $a_l=y$. If $S[t+1,T_U]$ hits
$\alpha$, its first such visit is at some $a_j$ with $j<l$, because
$t$ is the last visit to $y$; this visit erases $y$, which is never
visited again. If the intersection is empty, then
$\LE(S[0,T_U])=\alpha\oplus\LE(S[t,T_U])$. This proves
\eqref{eq:survival}.
Splitting at the last visit to \(y\) gives a bijection between paths with \(y_Q=y\) and pairs of paths satisfying the two conditions in \eqref{eq:singlemark}. Indeed, the second condition ensures that the concatenation point is the last visit to \(y\). Since path weights factor under concatenation, \eqref{eq:singlemark} follows from the definition of the Green bridge.
\end{proof}

Iterating Lemma~\ref{lem:surviving-entrance} at the marked vertices
gives an exact expansion for several boxes. This is the multiple-mark
decomposition of~\cite[Proposition~5.2]{BM}, in which we retain the
first-entrance condition. Fix pairwise disjoint cubes $Q_1,\ldots,Q_k$, none of which
contains $0$, put $y_0=0$, and fix $y_i\in\inn Q_i$. Let
$Z^j\sim\br_{D_N}^{y_j,y_{j+1}}$ for $0\le j<k$, and let $Z^k$ be a walk
from $y_k$ stopped on leaving $D_N$, all independent. Write
$n_j=|Z^j|$ for $j\leq k$ and $\gamma^j=\LE(Z^j)$ for $j< k$, and define
\begin{align}
 \mathcal H&=\bigcap_{i=1}^k\Bigl\{\gamma^{i-1}\cap
   \bigcup_{j=i}^kZ^j[1,n_j]=\varnothing\Bigr\},
 \label{eq:os-bridge-survival}\\
 \mathcal V&=\bigcap_{i=1}^k\bigl\{(\gamma^0\oplus\cdots\oplus
   \gamma^{i-1})\cap Q_i=\{y_i\}\bigr\}.
 \label{eq:os-bridge-first-entrance}
\end{align}

\begin{lemma}\label{lem:ordered-expansion}
With this notation,
\begin{equation}\label{eq:os-exact-ordered-expansion}
 \PP\{Q_1\to\cdots\to Q_k\}
 =\sum_{\substack{y_i\in\inn Q_i\\1\le i\le k}}
  \Bigl[\prod_{i=1}^kG_{D_N}(y_{i-1},y_i)\Bigr]
  \PP(\mathcal H\cap\mathcal V).
\end{equation}
On $\mathcal H\cap\mathcal V$, for every $1\le i\le k$,
\begin{equation}\label{eq:os-incoming-target-avoidance}
 \gamma^{i-1}\cap Q_i=\{y_i\},\qquad
 \gamma^{i-1}\cap Z^i[1,n_i]=\varnothing.
\end{equation}
\end{lemma}

\begin{proof}
Since $0\notin Q_i$, the first vertex $y_{Q_i}$ of $Q_i$ on $\gamma_N$
lies in $\inn Q_i$. On $\{Q_1\to\cdots\to Q_k\}$, the last visits
$t_1,\ldots,t_k$ of the walk to $y_{Q_1},\ldots,y_{Q_k}$ increase,
because these vertices appear in this order on $\gamma_N$. Split the
walk at $t_1<\cdots<t_k$. By \eqref{eq:survival} at each $t_i$, the
walk after $t_i$ avoids $\LE(S[0,t_i])$, so
$\LE(S[0,t_i])=\gamma^0\oplus\cdots\oplus\gamma^{i-1}$, and the pieces
satisfy $\mathcal H\cap\mathcal V$. Conversely, pieces satisfying
$\mathcal H\cap\mathcal V$ concatenate to a walk in the event, with
these marks. The splitting is unique, and the path weights factor as
in the proof of Lemma~\ref{lem:surviving-entrance}.
\end{proof}

The first condition in \eqref{eq:os-incoming-target-avoidance} is the
reason for marking surviving entrances. Without it, the two-arm
probability at a single vertex at scale $r$ is of order $r^{-\theta_d}$,
and summing over the order $L^{d-1}$ possible entrance vertices exceeds
the bound of Lemma~\ref{lem:sharp-inner-flux} below by the factor
$L^{b_d-1}$. The resulting estimate would not be uniform in the mesh.

\subsection{The estimate at a single box}
\label{sec:one-box}

For a cube $Q$, a vertex $y\in\inn Q$ and $r\ge1$, let $\gamma^\infty$ be
an infinite LERW and $S$ an independent simple random walk, both
started at $y$, and let $j_r$ and $\tau_r$ be their first exit indices
from $B(y,r)$. Set
\begin{equation}\label{eq:sharp-local-q}
 q_Q(y,r)=\PP(A_r^Q),\qquad
 A_r^Q=\bigl\{\gamma^\infty[1,j_r]\cap Q=\varnothing,\
 S[1,\tau_r]\cap\gamma^\infty[0,j_r]=\varnothing\bigr\}.
\end{equation}
For $Q=\{y\}$ this is, up to constants, the non-intersection
probability $\mathrm{Es}(r)$ of~\cite{BM,Shiraishi}; in dimension three,
$\mathrm{Es}(n)=cn^{\beta-2}[1+O(n^{-\kappa})]$ for every $n\ge1$
by~\cite[Theorem~4.3]{HLS}.

\begin{lemma}\label{lem:sharp-inner-flux}
There are constants $C_0,C<\infty$, depending only on $d$, such that
for every cube $Q$ of side $L\ge1$ and every $r\ge C_0L$,
\begin{equation}\label{eq:sharp-inner-flux}
 \sum_{y\in\inn Q}q_Q(y,r)\le CL^{d-2}\Bigl(\frac Lr\Bigr)^{\theta_d}.
\end{equation}
\end{lemma}
\begin{proof}
Let $\eta$ be a reference LERW in a domain of diameter comparable to
$r$, with $Q$ at distance comparable to $r$ from its starting point
and from the domain boundary.  We take radial LERW in dimension three
and chordal LERW in dimension two.  Write $\mathcal H_y$ for the event
that $y\in\inn Q$ is the first vertex of $Q$ visited by $\eta$.

The last-visit decomposition of Lemma~\ref{lem:surviving-entrance}
(and the loop-erased path-weight formula
\cite[Propositions~9.3.3 and~9.5.1]{LL}) gives
\[
  \PP(\mathcal H_y)=\kappa_y p_y,
  \qquad \kappa_y\asymp r^{2-d},
\]
where $p_y$ is the corresponding non-intersection probability for an
erased incoming arm and an independent raw outgoing arm, with the
erased arm meeting $Q$ only at $y$.  To compare $p_y$ with $q_Q(y,r)$, we use the separation estimates
with the additional requirement that the erased arm avoid $Q$.
Let $\mathrm{Sep}_r$ denote the event that the two arms are
separated at scale $r$, in the sense that the endpoint of each
arm is at distance at least $\varepsilon r$ from the other arm,
for some fixed $\varepsilon>0$. We claim that
\[
 \PP(A_r^Q\cap\mathrm{Sep}_r)\ge c\,\PP(A_r^Q),
\]
with constants independent of $Q,y,r$.

In dimension three, this follows from the history-uniform
separation estimate
\cite[equation~(6.15), in the proof of Theorem~6.1.5]{Shiraishi}.
Indeed, at scale $r/2$, we may augment the raw-walk history
by a closed lattice walk based at $y$ that visits every vertex
of $Q\cap\ZZ^3$. Requiring the erased continuation to avoid
this augmented history incorporates the avoidance of $Q$,
without changing the laws of the continuations.

In dimension two, the iteration of
\cite[Theorem~4.7 and Lemmas~4.8--4.9]{Masson2009} applies
conditionally on both exposed prefixes, from radius $r/2$ to $r$.
On histories whose erased prefix meets $Q$ only at $y$, the
conditional bad-event upper bounds remain valid with the additional
avoidance of $Q$, and the good-extension lower bounds are unchanged
because the outward wedges lie outside $B(y,r/4)$, which contains
$Q$ once $C_0\ge4\sqrt d$.

The laws of the local prefixes of the two arms in $p_y$ are
comparable to those in \eqref{eq:sharp-local-q}
\cite[Propositions~4.2, 4.4 and Corollary~4.5]{Masson2009}. On
$\mathrm{Sep}_r$, the erased arm can be extended through a corridor
to the starting point of $\eta$, and the raw arm through a disjoint
corridor to the target of $\eta$, with probability bounded below,
the erased continuation being sampled under its domain Markov law;
see \cite[equations~(6.21)--(6.22)]{Shiraishi} and
\cite[Proposition~3.5]{Masson2009}. Hence $p_y\ge c\,q_Q(y,r)$, that is,
\begin{equation}\label{eq:sharp-reference-comparison}
  q_Q(y,r)\le Cr^{d-2}\PP(\mathcal H_y),
  \qquad y\in\inn Q.
\end{equation}
The events $\mathcal H_y$ are disjoint as $y$ varies.  Consequently,
by the sharp ball-hitting estimates
\cite[Lemma~4.4]{HLS} for $d=3$ and
\cite[Theorem~6.1 and Corollary~6.3]{LV} for $d=2$,
\begin{align*}
 \sum_{y\in\inn Q}q_Q(y,r)
 \le Cr^{d-2}\PP(\eta\cap Q\ne\varnothing)\le Cr^{d-2}\Bigl(\frac Lr\Bigr)^{d-b_d}
  =CL^{d-2}\Bigl(\frac Lr\Bigr)^{\theta_d},
\end{align*}
as required.
\end{proof}

\subsection{Multiplying the local estimates}
\label{sec:box-dependence}

To bound $\PP(\mathcal H\cap\mathcal V)$ in
\eqref{eq:os-exact-ordered-expansion}, we keep only the conditions
\eqref{eq:os-incoming-target-avoidance}, restricted to balls around the
marked vertices. At $y_i$ they involve the terminal part of the erased
bridge $\gamma^{i-1}$ and the initial part of the raw path $Z^i$. Since
loop erasure does not commute with reversal pathwise, we use Lawler's
bijection $\mathcal R$ on paths from $y_{i-1}$ to $y_i$
\cite{LawlerGreen}, \cite[Lemma~3.1]{BM}: it preserves path weights and
the number of traversals of each directed edge, and
$\LE(\mathcal R\omega)=(\LE(\omega^R))^R$. Replacing every $Z^{i-1}$ by
$\mathcal RZ^{i-1}$ changes neither the law of the bridges nor their
raw traces, and turns the terminal part of $\gamma^{i-1}$ into the
reversal of the initial part of $\LE((Z^{i-1})^R)$, the erasure of a
bridge started at $y_i$. The local pairs at different marks are then
compared with independent pairs by the following estimate.

\begin{lemma}\label{lem:bridge-two-end}\label{lem:weighted-endpoints}
Let $U\subset\ZZ^d$ be finite, $Z\sim\br_U^{a,b}$ and $r_a,r_b\ge1$, and
suppose that $B(a,8r_a)$ and $B(b,8r_b)$ are disjoint subsets of $U$.
Let $Z_a$ be $Z$ stopped on leaving $B(a,r_a)$, and let $\widehat\gamma$
be $\LE(Z^R)$ stopped on leaving $B(b,r_b)$. Then, for all possible
values $\omega,\eta$,
\begin{equation}\label{eq:weighted}
 \PP(Z_a=\omega,\ \widehat\gamma=\eta)\le C\,
 \PP^a\bigl(S[0,T_{B(a,r_a)}]=\omega\bigr)\,
 \PP\bigl(\gamma^\infty[0,j_{r_b}]=\eta\bigr),
\end{equation}
where $\gamma^\infty$ is infinite LERW from $b$ and $C$ depends only on
$d$. In particular, $\PP(\widehat\gamma=\eta)\le
C\PP(\gamma^\infty[0,j_{r_b}]=\eta)$ whenever $B(b,8r_b)\subset U$ and
$a\notin B(b,8r_b)$.
\end{lemma}

\begin{proof}[Sketch of proof]
Given $\widehat\gamma=\eta$ with tip $v$, the first visit of $Z$ to
$\ran(\eta)$ is at $v$, and the part of $Z$ before this visit is a
walk from $a$ conditioned to first hit $\ran(\eta)$ at $v$ before
leaving $U$. Thus the conditional law of $Z_a$ is the Doob transform of
the walk law by $x\mapsto\PP^x\{S$ first hits $\ran(\eta)$ at $v$
before $T_U\}$, which is positive and harmonic on $B(a,8r_a)$; Harnack's
inequality bounds the density by $C$. The bound on the law of
$\widehat\gamma$ is \cite[Propositions~4.2, 4.4 and
Corollary~4.5]{Masson2009}, since $\LE(Z^R)$ is the erasure of a walk
from $b$ conditioned to hit $a$.
\end{proof}

For the cubes of one family, with centers $q_i$ and $q_0=0$, define
\begin{equation}\label{eq:sharp-target-scales}
 \rho_i=\begin{cases}
 \min\{\dist(q_i,\partial B(0,N)),\,|q_i-q_{i-1}|,\,|q_i-q_{i+1}|\},
 &i<k,\\[2pt]
 \min\{\dist(q_k,\partial B(0,N)),\,|q_k-q_{k-1}|\},&i=k.
 \end{cases}
\end{equation}
By the choice of the families, $\rho_i\ge A_0L$.

\begin{proposition}\label{prop:sharp-ordered-boxes}\label{prop:r3-ordered-boxes}
For distinct cubes $Q_1,\ldots,Q_k$ of one family,
\begin{equation}\label{eq:sharp-ordered-boxes}
 \PP\{Q_1\to\cdots\to Q_k\}\le C^k\prod_{i=1}^k
 g_d(q_{i-1},q_i)\,L^{d-2}\Bigl(\frac L{\rho_i}\Bigr)^{\theta_d},
\end{equation}
where $g_2(x,y)=\log\frac{CN}{L+|x-y|}$ and $g_3(x,y)=(L+|x-y|)^{-1}$.
The constants do not depend on $N$, $L$, $k$ or the cubes.
\end{proposition}

\begin{proof}
Put $R_i=c_1\rho_i$, with $c_1>0$ small and then $A_0\ge C_0/c_1$
large. Then the balls $B(y_i,8R_i)$ lie in $D_N$, the two balls at the
ends of each bridge are disjoint, $0\notin B(y_1,8R_1)$, and
$R_i\ge C_0L$; balls attached to nonadjacent bridges may overlap. In
\eqref{eq:os-exact-ordered-expansion}, bound $\PP(\mathcal H\cap\mathcal V)$
by the probability that \eqref{eq:os-incoming-target-avoidance} holds
inside $B(y_i,R_i)$ for every $i$. After the reversal above, these
conditions concern the pairs formed by the erased incoming prefix and
the raw outgoing prefix at each $y_i$, and each bridge contributes to
two consecutive pairs. Applying Lemma~\ref{lem:bridge-two-end} to the
independent bridges, with one constant per bridge, and noting that the
terminal walk $Z^k$ has the ordinary law, we obtain
\begin{equation}\label{eq:sharp-ordered-vertex-bound}
 \PP\{Q_1\to\cdots\to Q_k\}\le C^k\sum_{\substack{y_i\in\inn Q_i\\
 1\le i\le k}}\prod_{i=1}^kG_{D_N}(y_{i-1},y_i)\,q_{Q_i}(y_i,R_i).
\end{equation}
Since $|q_i-q_{i-1}|\ge A_0L$, the Green function satisfies
$G_{D_N}(y_{i-1},y_i)\le Cg_d(q_{i-1},q_i)$ uniformly in the entrance
vertices. The sums over the $y_i$ then factor, and
Lemma~\ref{lem:sharp-inner-flux} with $R_i\asymp\rho_i$ completes the
proof.
\end{proof}

\subsection{Spatial summation}
\label{sec:covering-moments}\label{sec:sharp-chain-sum}

Proposition~\ref{prop:sharp-ordered-boxes} is the box analogue of the
multi-point estimates used in \cite[Section~5]{BM} and
\cite[Section~8.1]{Shiraishi}, with a point replaced by a box of side
$L$. After rescaling the centers by $L$, the summation over positions
proceeds as there. We record the form we use, since the boundary
distance enters through the same minimum as the neighboring targets.
Let $R\ge1$, $V_R=B(0,R)\cap\ZZ^d$ and $h_R(x)=1+\dist(x,\partial B(0,R))$,
and put
\[
 g_{2,R}(x,y)=\log\frac{AR}{1+|x-y|},\qquad
 g_{3,R}(x,y)=\frac1{1+|x-y|},
\]
with $A$ a large constant. For $x_0\in B(0,R)$, let
$s_i=1+|x_i-x_{i-1}|$, $d_i=\min\{s_i,s_{i+1},h_R(x_i)\}$ for $i<k$,
and $d_k=\min\{s_k,h_R(x_k)\}$.

\begin{lemma}\label{lem:chain-kernel}
For $d\in\{2,3\}$ and $0\le\theta<1$,
\begin{equation}\label{eq:sharp-chain-kernel}
 \sum_{x_1,\ldots,x_k\in V_R}\prod_{i=1}^k
 g_{d,R}(x_{i-1},x_i)\,d_i^{-\theta}\le[C_\theta R^{2-\theta}]^k.
\end{equation}
The same holds for any set of centers with bounded density in unit
cubes in place of $V_R$.
\end{lemma}

\begin{proof}[Sketch of proof]
Expand $d_i^{-\theta}\le s_i^{-\theta}+s_{i+1}^{-\theta}+h_R(x_i)^{-\theta}$
(without the middle term for $i=k$), which gives at most $3^k$ terms.
In each term the vertex $x_i$ carries a factor
$s_i^{-a_i}h_R(x_i)^{-z_i}$ with $z_i\le\theta<1$, $a_i+z_i\le2\theta<2$
and $\sum_i(a_i+z_i)=k\theta$. Sum over $x_k,x_{k-1},\ldots,x_1$
successively, using
\[
 \sup_{y\in B(0,R)}\sum_{x\in V_R}g_{d,R}(y,x)(1+|x-y|)^{-a}h_R(x)^{-z}
 \le C_{a,z}R^{2-a-z},\qquad z<1,\ a+z<2,
\]
which follows by decomposing into dyadic annuli around $y$ and
boundary layers; the condition $z<1$ makes the boundary layers
summable, and in dimension two the logarithm costs only a constant.
\end{proof}

\begin{proof}[Proof of Theorem~\ref{thm:sharp-fixed-scale}]
Let $X$ be the number of cubes of one family hit by $\gamma_N$. Put
$R=N/L$ and $x_i=q_i/L$. Then $L\,g_3(q_{i-1},q_i)\le Cg_{3,R}(x_{i-1},x_i)$,
$g_2(q_{i-1},q_i)\le Cg_{2,R}(x_{i-1},x_i)$, and $\rho_i/L$ is comparable
to the minimum $d_i$ of Lemma~\ref{lem:chain-kernel}. Insert
\eqref{eq:sharp-ordered-boxes} into
\eqref{eq:os-factorial-moment-expansion}, drop the constraint that the
cubes be distinct, and apply Lemma~\ref{lem:chain-kernel} with
$\theta=\theta_d$ to the rescaled centers:
\begin{equation}\label{eq:sharp-factorial-bound}
 \EE\binom Xk\le[CR^{2-\theta_d}]^k=[C(N/L)^{b_d}]^k.
\end{equation}
This is \eqref{eq:box-outline-moments}. Together with the bounded
number of cubes near $0$ and Lemma~\ref{lem:r3-boundary-count}, it
proves \eqref{eq:sharp-fixed-scale}.
\end{proof}

\subsection{Covering content over all scales}
\label{sec:covering-scales}

Recall $M_{\alpha,\delta}$ from \eqref{eq:box-content-definition}. The fixed-scale bound becomes stronger at small
scales when $\alpha>b_d$, which allows us to sum over scales.

\begin{lemma}\label{lem:tail-upgrade}
For every $b_d<\alpha<d$,
\begin{equation}\label{eq:sharp-all-scale-tail}
 \sup_{0<\delta\le1}\PP_\delta\{M_{\alpha,\delta}(K^\delta)>u\}
 \le C_\alpha\exp\bigl\{-c_\alpha u^{(d-b_d)/(d-\alpha)}\bigr\},
 \qquad u\ge1.
\end{equation}
Consequently, for every $\alpha>b_d$ and $t>0$,
$\sup_{0<\delta\le1}\EE_{\PP_\delta}\exp\{tM_{\alpha,\delta}(K^\delta)\}
<\infty$. Both statements also hold when the supremum in
\eqref{eq:box-content-definition} is taken over all $r\in[\delta,1]$.
\end{lemma}

\begin{proof}
Since $N_r(K^\delta)\le C_Dr^{-d}$, the event
$\{r^\alpha N_r(K^\delta)>u\}$ is empty unless
$r\le r_*(u)=(C_D/u)^{1/(d-\alpha)}$. At the $j$th dyadic scale below
$r_*(u)$, apply the rescaled form of
Theorem~\ref{thm:sharp-fixed-scale} with parameter
$ur^{b_d-\alpha}\ge cu^{(d-b_d)/(d-\alpha)}2^{j(\alpha-b_d)}$, and sum
over $j$. Since $(d-b_d)/(d-\alpha)>1$, integrating the tail gives the
exponential moments; for $\alpha\ge d$ the content is bounded. Finally,
$r^\alpha N_r\le Cs^\alpha N_s$ for a dyadic $s\in[r,2r]$.
\end{proof}

\begin{proof}[Proof of Theorem~\ref{thm:covering-input}]
The fixed-scale bound \eqref{eq:sharp-fixed-scale-input} is
Theorem~\ref{thm:sharp-fixed-scale} after rescaling, and \eqref{eq:EB}
follows from Lemma~\ref{lem:tail-upgrade}.
\end{proof}

\appendix

\section{Spectral support and capacitary constructions}\label{app:spectral}

In this appendix, we prove Propositions~\ref{prop:all-support}
and~\ref{prop:thresholds}. We use the capacity theory of~\cite[Chapters~4 and~8]{MP}
and the Sobolev formulation in~\cite[Section~2]{BBcapacity};
see also~\cite{GR} for related spectral results.

\subsection{Brownian containment and capacitary obstacles}

Throughout the appendix, Brownian motion has generator
$\Delta/(2d)$. The first lemma shows that the classical limiting arc
contains the separating points of a coupled Brownian path.

\par\nopagebreak[4]
\begin{lemma}\label{lem:containment}
There is a coupling of the trace of the ordinary LERW scaling limit $K$ and stopped
Brownian trace $\mathcal B=B[0,T_D]$ with $K\subset\mathcal B$ such that
every $z\in\mathcal B\setminus\{0,B_{T_D}\}$ whose deletion separates
$0$ from $B_{T_D}$ in $\mathcal B$ belongs to $K$.
\end{lemma}
\begin{proof}
The trace/loop-soup identity \cite[Theorem~1.1]{SS} supplies the inclusion;
disintegration over the trace supplies a parametrized Brownian motion.
By \cite[Theorem~1.2]{SS}, $K$ is an arc between its endpoints, hence
contains every separating point, as observed after Theorem~1.4 there.
\end{proof}

We use the capacity associated with the Dirichlet form of $A_U$
and, in dimension three, the Newtonian capacity $\capn$ of the whole
space, which satisfies $\capn(z+rF)=r\capn(F)$. A function satisfies
the Dirichlet condition on an obstacle if it vanishes quasi-everywhere
there, that is, outside a set of capacity zero.

We next record two properties of the principal eigenvalue of slit
domains.

\par\nopagebreak[4]
\begin{lemma}\label{lem:capacity-spectrum}
If $K\cap D$ contains an interior compact set of positive relative
capacity, then $\lambda_1(D\setminus K)>\lambda_1(D)$.
For compact $K\subset\overline D$ and
$K^\varepsilon=\{x:\dist(x,K)\leq\varepsilon\}$,
\begin{equation}\label{eq:neighborhood-spectrum}
 \lambda_1(D\setminus K^\varepsilon)\downarrow
 \lambda_1(D\setminus K)\qquad(\varepsilon\downarrow0).
\end{equation}
For classical LERW the strict inequality holds almost surely in both dimensions.
\end{lemma}

\begin{proof}
Equality of the first eigenvalues would make a minimizer on the slit
domain a multiple of the positive ground state on $D$, which cannot
vanish quasi-everywhere on an interior set of positive capacity.
Neighborhood continuity follows from domain monotonicity and the
Rayleigh principle, since every compactly supported smooth test
function on $D\setminus K$ avoids $K^\varepsilon$ for small
$\varepsilon$. A planar arc has an interior subcontinuum of positive
capacity. In dimension three, $\dim_HK>1$ by~\cite[Theorem~1.4]{SS},
so an interior compact part has positive capacity by Frostman's
criterion. This proves the last assertion.
\end{proof}
If $K$ contains a fine array of obstacles of sufficiently large
capacity, Brownian motion hits $K$ quickly, which bounds the principal
eigenvalue from below.

\par\nopagebreak[4]
\begin{lemma}
\label{lem:cap-mesh}
In dimension three, let $K\subset\overline D$ be compact and fix
$A,L,\kappa>0$. Suppose $K$ contains compact sets
$F_i\subset B(z_i,Ar)$ with $\capn(F_i)\geq\kappa r$, and every point
of $D$ more than $Lr$ from the sphere is within $Lr$ of some $z_i$.
There are $c,r_0>0$, depending only on these parameters, such that
$\lambda_1(D\setminus K)\geq cr^{-2}$ for $r<r_0$.
\end{lemma}

\begin{proof}
Within $Lr$ of the boundary, Brownian motion exits $D$ by time $Tr^2$
with uniformly positive probability. Elsewhere, choose a nearby
$F_i$ and enclose it and the starting point in a ball of radius $Mr$,
with $M$ a fixed large constant. Its Green kernel is bounded below
by $c/r$ on these points. The equilibrium-measure formula therefore
gives a hitting probability at least
$(c/r)\capn(F_i)\ge c\kappa$ before leaving the ball.
Choose fixed $T$ large enough that a late exit costs at most half
this probability. Thus Brownian motion hits $K$ or leaves $D$ by
$Tr^2$ with probability at least some fixed $p>0$, uniformly in its
starting point. The killed semigroup has $L^\infty$ norm at most
$1-p$ at that time; symmetry and interpolation give the same
$L^2$ bound, hence $\lambda_1(D\setminus K)\ge cr^{-2}$.
\end{proof}

\subsection{Positive support at arbitrarily large eigenvalues}

This subsection is dedicated to Proposition~\ref{prop:all-support}. We obtain traces with arbitrarily large slit-domain eigenvalue by
confining the coupled Brownian motion to suitable regions. In
dimension three, the obstacles come from the following construction,
which joins Brownian pieces in small cubes while keeping enough global
cutpoints.

\par\nopagebreak[4]
\begin{lemma}
\label{lem:rooms}
There are a bounded cube $V\subset\RR^3$, constants
$\kappa>0$, $0<u_-<u_+<1$, and disjoint entrance and exit gate balls
in $V$ with the following property. Place finitely many rooms
$z_i+rV$ strictly inside $D$, at mutual distances at least a large fixed
multiple of $r$, to accommodate their entrance and exit stubs. Then the
rooms can be connected in order from $0$ to the sphere so that, with
positive Brownian probability, every room contains a compact set $F_i$
consisting of separating global cutpoints and their limits, with
\[
 F_i\subset z_i+rV,\qquad \capn(F_i)\ge\kappa r
 \quad\text{for every room }i.
\]
The template constants are independent of $r$ and the number of rooms;
the probability need not be uniform.
\end{lemma}

\begin{proof}
Consider a unit-time Brownian piece confined to a sufficiently large
cube $V$, with its starting point integrated over an interior ball.
The cutpoint-dimension theorem~\cite[Theorem~1.1 and the discussion
before Corollary~1.2]{Cutpoints} and Frostman's criterion give a
positive-measure event on which the closure of the cutpoints in a
fixed interval $[u_-,u_+]\subset(0,1)$ has capacity at least
$\kappa>0$. By point polarity and reversal, small gate balls around
the initial and terminal points can be chosen to avoid, respectively,
the portions after $u_-$ and before $u_+$. Attach disjoint polygonal
stubs from these gates to the exterior of $V$. Fixed polygonal arcs
are polar in three dimensions; using a countable rational family
and then shrinking neighborhoods, we obtain one deterministic pair
of gate-and-stub tubes $T^-,T^+$ and a positive-measure event satisfying
\[
 T^-\cap\omega[u_-,1]=\varnothing,
 \qquad T^+\cap\omega[0,u_+]=\varnothing.
\]
The gate balls are contained in the corresponding tubes.

Place scaled copies of this construction in the prescribed rooms. Join their outer stubs
in order by disjoint narrow tubes, avoiding all nonincident rooms;
such polygonal routes exist in three dimensions. Include an initial
connector from $0$ and a final connector through the sphere.
In each room impose the scaled positive-measure event, and in each
connector require Brownian motion to stay in its tube for a fixed
time and reach the next gate. The killed transition densities in the
tubes are strictly positive, and each room event has a nonzero
endpoint density, so the Markov property gives positive probability
to the finite concatenation and final exit. The gate restrictions
keep the incoming path away from each retained cutpoint's future,
and the outgoing path away from its past; disjointness of the other
rooms and connectors then makes these cutpoints global. Their
closures still have capacity at least $\kappa r$, by scaling.
\end{proof}

\begin{proof}[Proof of Proposition~\ref{prop:all-support}]
In dimension three put room centers on a lattice of spacing a fixed
large multiple of $r$, omitting a boundary layer and a few rooms near
$0$. Lemmas~\ref{lem:rooms} and~\ref{lem:containment} place positive-capacity
obstacles inside $K$ with positive probability. The centers satisfy
Lemma~\ref{lem:cap-mesh} after enlarging $L$, so
$\lambda_1(D\setminus K)\ge cr^{-2}$.

In the plane, confine Brownian motion to a narrow corridor around a
simple spiral from $0$ to the circle visiting an $r$-net. This has
positive probability. Containment forces $K$ to cross every corridor
cross-section, making $K\cup D^c$ connected and $Cr$-dense.
Beurling's estimate~\cite{Be} then gives a uniformly positive probability to
hit this set before leaving a ball of radius $Lr$. A ball-exit tail
bound puts this hit before $Tr^2$ with uniformly positive probability.
The semigroup argument above again gives
$\lambda_1(D\setminus K)\ge cr^{-2}$. Let $r\downarrow0$.
\end{proof}

\subsection{The lower edges of the spectral support}

We now turn to Proposition~\ref{prop:thresholds}. The lower edge of the support is found by comparison with a radial
segment: the segment is polar in dimension three, and it is the
extremal slit in dimension two.

\begin{proof}[Proof of Proposition~\ref{prop:thresholds}]

In dimension three, strictness above $\lambda_1(D)$ follows from
Lemma~\ref{lem:capacity-spectrum}. A straight radial segment $\gamma_*$ is polar,
so removing it leaves $H^1_0(D)$ and the spectrum unchanged.
By \eqref{eq:neighborhood-spectrum},
$\lambda_1(D\setminus\gamma_*^\varepsilon)\downarrow\lambda_1(D)$.
Brownian motion has positive probability of staying in this thin tube
until exit. Containment gives $K\subset\gamma_*^\varepsilon$ and hence
$X\leq\lambda_1(D\setminus\gamma_*^\varepsilon)$ on that event.

In the plane, a radial arc $K$ meets every circle about $0$.
For $f\in C_c^\infty(D\setminus K)$ put
$v(r,\theta)=\sqrt r\,f(re^{i\theta})$. Each angular slice vanishes
at a point, so the Dirichlet Poincar\'e inequality on an interval of
length $2\pi$ gives $\int|\partial_\theta v|^2\ge\frac14\int|v|^2$.
Integration by parts in $r$ and the interval $(0,1)$ inequality yield
\begin{equation}\label{eq:planar-slit-energy}
 \int_D|\nabla f|^2
 =\int_0^1\!\int_0^{2\pi}
 \left[|\partial_rv|^2+
 \frac{|\partial_\theta v|^2-|v|^2/4}{r^2}\right]\dd\theta\dd r
 \ge\pi^2\int_D|f|^2.
\end{equation}
Thus $X\ge\pi^2/4$. For a radial segment $\gamma_*$, equality is attained
by $r^{-1/2}\sin(\pi r)\sin(\theta/2)$, with $0<\theta<2\pi$
measured from the slit. The same thin-tube event as above and
\eqref{eq:neighborhood-spectrum} give $\essinf X=\pi^2/4$.

Moreover, equality in \eqref{eq:planar-slit-energy} determines the
arc. The identity and both nonnegative remainders extend
to $H^1_0(D\setminus K)$ by density (the radial Hardy inequality
controls $v/r$). Equality for a nonnegative ground state forces
$v(r,\theta)=\sin(\pi r)a(\theta)$, and equality in the angular
inequality forces $a(\theta)=c|\sin((\theta-\theta_0)/2)|$.
This function is positive away from one radial segment. Its
quasi-continuous representative vanishes on $K$ outside a set of
capacity zero; any subarc away from that segment has positive planar
capacity. Hence $K$ must be the segment. Radial ${\rm SLE}_2$ is almost
surely not a segment, since its Loewner driving Brownian motion is not
constant. Therefore $\PP(X>\pi^2/4)=1$.
\end{proof}

\end{document}